\documentclass[aos]{imsart}

\RequirePackage{amsthm,amsmath,amsfonts,amssymb}
\RequirePackage[numbers]{natbib}
\RequirePackage[colorlinks,citecolor=blue,urlcolor=blue]{hyperref}
\RequirePackage{graphicx}

\startlocaldefs
\theoremstyle{plain}

\newtheorem{theorem}{Theorem}[section]
\newtheorem{lemma}[theorem]{Lemma}
\newtheorem{proposition}[theorem]{Proposition}
\newtheorem{remark}[theorem]{Remark}
\theoremstyle{remark}

\newtheorem{example}{Example}

\newtheorem{assumption}{Assumption}
\usepackage{amsmath,amssymb,amsthm}
\usepackage{esint}
\usepackage{mathtools} 
\usepackage{graphicx}
\usepackage{MnSymbol}
\usepackage{mathtools} 
\usepackage[bottom]{footmisc}
\usepackage{microtype}

\usepackage{bm}
\usepackage{dsfont}
\usepackage{booktabs} % nice tables

\newcommand{\E}{\mathbb{E}}

\newcommand{\eps}{\varepsilon}

\renewcommand{\le}{\leqslant}
\renewcommand{\ge}{\geqslant}
\renewcommand{\leq}{\leqslant}
\renewcommand{\geq}{\geqslant}

\newcommand{\la}{\langle}
\newcommand{\ra}{\rangle}

\renewcommand{\d}{\mathrm{d}}

\DeclareMathOperator{\divg}{div}

\renewcommand{\tilde}{\widetilde}

\renewcommand{\hat}{\widehat}

\makeatletter
\newsavebox\myboxA
\newsavebox\myboxB
\newlength\mylenA
\newcommand*\mybar[2][0.75]{%
    \sbox{\myboxA}{$\m@th#2$}%
    \setbox\myboxB\null% Phantom box
    \ht\myboxB=\ht\myboxA%
    \dp\myboxB=\dp\myboxA%
    \wd\myboxB=#1\wd\myboxA% Scale phantom
    \sbox\myboxB{$\m@th\overline{\copy\myboxB}$}%  Overlined phantom
    \setlength\mylenA{\the\wd\myboxA}%   calc width diff
    \addtolength\mylenA{-\the\wd\myboxB}%
    \ifdim\wd\myboxB<\wd\myboxA%
       \rlap{\hskip 0.5\mylenA\usebox\myboxB}{\usebox\myboxA}%
    \else
        \hskip -0.5\mylenA\rlap{\usebox\myboxA}{\hskip 0.5\mylenA\usebox\myboxB}%
    \fi}
\makeatother

\newcommand{\p}{{\mathbb{P}}}
\newcommand{\e}{{\mathbb{E}}}
\newcommand{\Reals}{\mathbb{R}}

\newcommand\rrac[2][r]{%
  \ifx r#1 (#2)\else
  \ifx s#1 [#2]\else
  \ifx c#1 \{#2\}\else
  \ifx v#1 |#2|\else
  \ifx a#1 \langle#2\rangle\else  
  \mathrm{Illegal~option}%
  \fi\fi\fi\fi\fi
}

\newcommand\Brac[2][r]{%
  \ifx r#1 \Big(#2\Big)\else
  \ifx s#1 \Big[#2\Big]\else
  \ifx c#1 \Big\{#2\Big\}\else
  \ifx v#1 \Big|#2\Big|\else
  \ifx a#1 \Big\langle#2\Big\rangle\else  
  \mathrm{Illegal~option}%
  \fi\fi\fi\fi\fi
}

\newcommand\brac[2][r]{%
  \ifx r#1 \big(#2\big)\else
  \ifx s#1 \big[#2\big]\else
  \ifx c#1 \big\{#2\big\}\else
   \ifx v#1 \big|#2\big|\else
  \ifx a#1 \big\langle#2\big\rangle\else    
  \mathrm{Illegal~option}%
  \fi\fi\fi\fi\fi
}

\endlocaldefs

\begin{document}

\begin{frontmatter}
%%%%%%%%%%%%%%%%%%%%%%%%%%%%%%%%%%%%%%%%%%%%%%
%%                                          %%
%% Enter the title of your article here     %%
%%                                          %%
%%%%%%%%%%%%%%%%%%%%%%%%%%%%%%%%%%%%%%%%%%%%%%
\title{Performance of Bayesian linear regression \\in a model with mismatch}
%\title{A sample article title with some additional note\thanksref{T1}}
\runtitle{Performance of Bayesian linear regression in a model with mismatch}
%\thankstext{T1}{A sample of additional note to the title.}

\begin{aug}
%%%%%%%%%%%%%%%%%%%%%%%%%%%%%%%%%%%%%%%%%%%%%%%
%% Only one address is permitted per author. %%
%% Only division, organization and e-mail is %%
%% included in the address.                  %%
%% Additional information can be included in %%
%% the Acknowledgments section if necessary. %%
%%%%%%%%%%%%%%%%%%%%%%%%%%%%%%%%%%%%%%%%%%%%%%%
\author[A]{\fnms{Jean} \snm{Barbier}\ead[label=e1,mark]{jbarbier@ictp.it, msaenz@ictp.it}},
\author[B]{\fnms{Wei-Kuo} \snm{Chen}\ead[label=e2,mark]{wkchen@umn.edu}}\\
\author[C]{\fnms{Dmitry} \snm{Panchenko}\ead[label=e3,mark]{panchenk@math.toronto.edu}}
\and
\author[A]{\fnms{Manuel} \snm{S\'aenz}\ead[label=e1,mark]{jbarbier@ictp.it, msaenz@ictp.it}}
%%%%%%%%%%%%%%%%%%%%%%%%%%%%%%%%%%%%%%%%%%%%%%
%% Addresses                                %%
%%%%%%%%%%%%%%%%%%%%%%%%%%%%%%%%%%%%%%%%%%%%%%
\address[A]{International Center for Theoretical Physics, Trieste, Italy, \printead{e1}}

\address[B]{School of Mathematics, University of Minnesota, Minneapolis, United States, \printead{e2}}

\address[C]{Department of Mathematics, University of Toronto, Canada, \printead{e3}}
\end{aug}

\begin{abstract}
In this paper we analyze, for a model of linear regression with gaussian covariates, the performance of a Bayesian estimator given by the mean of a log-concave posterior distribution with gaussian prior, in the high-dimensional limit where the number of samples and the covariates' dimension are large and proportional. Although the high-dimensional analysis of Bayesian estimators has been previously studied for Bayesian-optimal linear regression where the correct posterior is used for inference, much less is known when there is a mismatch. Here we consider a model in which the responses are corrupted by gaussian noise and are known to be generated as linear combinations of the covariates, but the distributions of the ground-truth regression coefficients and of the noise are unknown. This regression task can be rephrased as a statistical mechanics model known as the Gardner spin glass, an analogy which we exploit. Using a leave-one-out approach we characterize the mean-square error for the regression coefficients. We also derive the log-normalizing constant of the posterior. Similar models have been studied by Shcherbina and Tirozzi and by Talagrand, but our arguments are much more straightforward. An interesting consequence of our analysis is that in the quadratic loss case, the performance of the Bayesian estimator is independent of a global ``temperature'' hyperparameter and matches the ridge estimator: sampling and optimizing are equally good.
\end{abstract}

\end{frontmatter}

\section{Introduction and set-up}
Linear regression with random covariates in the high-dimensional regime, where both the number of data points and their dimension scale proportionally, is a paradigmatic problem for modern inference. It has been studied mostly from two complementary perspectives: $(i)$ In Bayesian statistics where, due to important simplifications that occur in this setting, analyses have focused on the case of the estimator obtained as the expectation of the ``true'' posterior distribution of the regression coefficients \cite{barbier2020mutual,barbier2016mutual,barbier2019optimal,reeves2016replica}. This is usually referred to as the \emph{Bayesian-optimal estimator}. But also on \emph{sparse} Bayesian regression, see \cite{castillo2015bayesian} and the section below on related works. $(ii)$ Another large body of literature has studied instead the performance of optimization procedures in the context of robust statistics and M-estimation, where an estimator is obtained as the mode of a log-concave posterior measure or, equivalently, as the minimizer of a convex cost function \cite{9173999, bean2013optimal,el2018impact,advani2016statistical,advani2016equivalence,bradic2016robustness,sur2019modern,celentano2019fundamental,donoho2016high,gerbelot2020asymptotic2,gerbelot2020asymptotic,aubin2020generalization,pmlr-v40-Thrampoulidis15,thrampoulidis2018precise,taheri2020fundamental}. 

Our work focuses on Bayesian inference outside the optimal setting, for which much less is known. The lack of perfect knowledge of the model generating the data forces one to make certain inaccurate prior assumptions on the distribution of the regression coefficients as well as on the form of the model. This is, in a sense, more realistic and of practical interest. One natural setting for this, is to consider Bayesian linear regression with some sources of mismatch. Due to the lack of certain simplifying symmetries inherent to the optimal setting of inference, known as the ``Nishimori identities'' in physics \cite{contucci2009spin,NishimoriBook01,barbier2020strong}, even  a small mismatch can result in several theoretical challenges. 

\subsection{Our model}
In the present work, we aim to study a mismatched linear regression model which is as follows.
Let $x^*=(x_i^*)_{i\leq N}$ be an unknown ground-truth of regression coefficients. Assume that 
\begin{align}\label{convNorm}
\lim_{N\to\infty}\e \Bigl|\frac{\|x^*\|}{\sqrt N}-\sqrt\gamma\Bigr|^2=0  \quad \mbox{and}\quad  \e\Big(\frac{ \|x^*\|}{\sqrt N}\Big)^{4}\le C<\infty 
\end{align}
for some deterministic $\gamma\geq 0$ and fixed $C$ independent of $N$. Let $\alpha>0$ be a fixed parameter and let $M=\lfloor\alpha N\rfloor$\footnote{To be more precise, we might take $M=\alpha N+o(N)$. But tracking this correction makes no difference.}; $\|\,\cdot\,\|$ is the usual $L_2$ norm. Consider a matrix $G=(g_{ki})_{k\leq M,i\le N}$ with i.i.d. standard normal entries. The columns in $G$ are understood as random covariate vectors, while the rows correspond to independent samples.
% \begin{align}\label{SkAj}
%   S_k(x)=(\bar G x)_k=\frac{1}{\sqrt{N}}\sum_{i\le N}g_{ki}x_i.
% \end{align}
%
 The responses $y=(y_k)_{k\le M}$ are obtained according to the linear model
\begin{align}
y_k=S_k(x^*) + z_k \label{truemodel}
\end{align}
for $S_k(x^*):=(\bar G x^*)_k,$
where $\bar G:=N^{-1/2}G$ and the variables $z=(z_k)_{k\le M}\stackrel{\rm{i.i.d.}}{\sim}\mathcal{N}(0,\Delta_*)$ for some $\Delta_*>0$ which gives the strength of the noise; $\mathcal{N}(m,\sigma^2)$ is the normal law of mean $m$ and variance $\sigma^2$. Here we will consider the regression task of recovering $x^*$ from the data $\mathcal{D}:=(y,\bar G)$ with two sources of mismatch:
\begin{enumerate}
    \item The distribution of $x^*$ is assumed to be unknown. Thus, we take the prior to be that of a gaussian vector with independent coordinates of mean $0$ and variance $1/(2\kappa)>0$.\\
    \item In the true model generating the data \eqref{truemodel}, the noise is given by i.i.d. gaussian random variables of variance $\Delta_*$. But its distribution is also assumed to be unknown. We therefore consider a factorized log-concave noise distribution with density proportional to $\prod_{k\le M} \exp u(x_k)$. Here $u$ is a non-positive concave function on $\mathbb{R}$ satisfying some further technical conditions given in Assumption \ref{ass:growth_u} below.
\end{enumerate}
Under these sources of mismatch, the Bayesian posterior distribution is given by
\begin{align}
P(x\in A \mid \mathcal{D})	=\frac{1}{{Z}(\mathcal{D})} \int_{A}\exp \Big(\sum_{k\le M} u \big(S_k(x)-y_k\big)- \kappa \|x\|^2\Big) \, dx\label{posterior}
\end{align}
for any measurable subset $A$ of $\mathbb{R}^N$, and ${{Z}(\mathcal{D})}$ is the normalization constant, also called the partition function in the terminology of statistical physics.
Our aim is to study the performance of the Bayesian estimator $\hat{x}$ defined through the mean of this mismatched Bayesian posterior distribution:
\begin{align}\label{estimator}
 \hat x (\mathcal{D}):= \int_{\mathbb{R}^N} x\, P(dx \mid \mathcal{D}).
\end{align}

\subsection{Main results and contributions}
Our main contributions are establishing the asymptotics of the log-normalizing constant and the performance of estimator \eqref{estimator}. Informally, we prove that, whenever the four-dimensional system (\ref{CPeqQ})--(\ref{CPeqRb}) presented below admits a unique solution, the log-normalizing constant converges to some non-random quantity $F(q,\rho)$, that is
\begin{gather*}
\lim_{N\to\infty} \e\Bigl|\frac1 N \ln {Z}(\mathcal{D}) -(-\kappa \gamma + F(q,\rho))\Bigr|^2=0,
\end{gather*}
and the mean-square error converges to some non-random quantity $q$, that is
\begin{align}\label{add:eq-16}
	\lim_{N\to\infty}\frac1N \E \|\hat x -x^*\|_2^2  = q.
\end{align}
Here $(q,\rho)$ is the unique critical point of the real-valued function $F$ with domain over the set $\{(q,\rho):0\leq q\leq \rho<\infty\}$ defined in \eqref{Fdef} below.

There are two major findings that follow from our results. First of all, somewhat surprisingly \eqref{add:eq-16} implies that the performance of the Bayesian estimator with gaussian prior is insensitive to the statistical properties of the coefficients $x^*$ other than its empirical second moment. The second major finding regards our main running example for our results which will play a special role given its tractability in the analysis and its importance in statistics. This example will correspond to the model where the function $u(s)$ is of the form $-s^2/(2\Delta)$ for some $\Delta>0$. Notice that this model is still mismatched since $\Delta$ is not necessarily equal to the $\Delta_*$ appearing in the model generating the data. When we convert the posterior distribution into a \emph{Gibbs measure} parametrized by the {\it inverse temperature} $\beta$ we discover that, for quadratic $u$, the mean-square error is independent of the inverse temperature $\beta$ in the high-dimensional limit. This means that when the performance is evaluated using a square loss, {Bayesian regression with gaussian prior and ridge regression have the same performance.}
In other words, sampling and optimizing are equivalent in this setting. One may argue that the closedness between the mean and the mode of the posterior distribution with gaussian prior is mainly due to the isotropy of the covariates. This intuition makes perfect sense in the classical regime, where a lot of data is accessible $M\gg N$ when maximum likelihood estimation is optimal \cite{wasserman2004all}. But in the high-dimensional regime where  both $M$ and $N$ are large and comparable, this feature becomes much less evident. To the best of our knowledge, this fact was not rigorously obtained prior to our work. Exploring the generality of this equivalence between optimization and sampling procedures is an interesting open direction for future research. See \cite{advani2016statistical,aubin2020generalization} for related connections between M-estimation and Bayesian inference from statistical mechanics heuristics.

The approach in this paper relies primarily on a connection between our regression task and a generalized version of the Shcherbina-Tirozzi (ST) mean-field spin glass model \cite{shcherbina2003rigorous}. We study the  limits of the corresponding log-partition function and some critical quantities called the overlaps. By utilizing the so-called smart path method, analogous results were intensively explored by Talagrand \cite{Talagrand2011spina} assuming that $u$ grows at most linearly. Whereas the machinery therein relies on quite heavy analysis, the advantage of our approach is methodological. In addition to being able to relax the growth condition for $u$ from linear to quadratic, our analysis is also rather simple and self-contained. It is based on methods coming from the mathematical physics of mean-field spin glasses: mainly a rigorous version of the cavity method \cite{Talagrand2011spina,panchenko2013sherrington} also called the leave-one-out method in statistics (see Sections~\ref{sec:5} and \ref{sec:freeenergy} for details). We expect for the simplicity of our approach to allow us to extend the analysis to generalized linear models in mismatched settings, for which results similar to the present ones exist only in the Bayesian-optimal setting of inference \cite{barbier2019optimal}. This is left for future work.

\subsection{Related works}

%The mode of this probability measure would correspond instead to a much more studied M-estimator with ridge regularisation %\cite{el2013robust,donoho2016high,el2018impact}. 

In recent years, many works on mismatched regression have appeared: on non-rigorous statistical mechanics techniques such as the replica method \cite{MezardParisi87b,MezardMontanari09} like, e.g., in \cite{advani2016statistical}; on the analysis of approximate message-passing algorithms, like in \cite{bradic2016robustness,sur2019modern,donoho2016high,gerbelot2020asymptotic}; and on Gordon's convex min-max theorem \cite{pmlr-v40-Thrampoulidis15,thrampoulidis2018precise}. All these approaches strongly rely on the fact that the estimator is obtained as a unique minimizer of a convex function and therefore do not generalize easily to mismatched Bayesian settings. The leave-one-out method was also used in \cite{el2013robust,el2018impact} but, like in the aforementioned references, only for the analysis of (non-Bayesian) M-estimators.

Complementary to our setup is a recent work \cite{mukherjee2021variational} on the mismatched Bayesian linear regression that derives conditions under which the ``naive mean-field approximation'' to the log-partition function is valid, and provides an infinite-dimensional variational formula for it (see \cite{ray2021variational} for a related setting). Last but not least, Bayesian linear regression has also been considered in some different high-dimensional scaling regimes than the one presented here: in settings with sparse ground-truth vectors of regression coefficients with a possibly vanishing fraction of non-zero entries. Reference \cite{castillo2015bayesian} studies in detail the properties of the posterior distribution with sparsity inducing priors in this case, see also \cite{bottolo2010evolutionary,george2000variable,george2000calibration,ishwaran2005spike,mitchell1988bayesian,scott2010bayes,yuan2005efficient,martin2017empirical,abramovich2010map,arias2014estimation} for related prior works on sparse linear regression. It has recently been understood that in the very sparse regime, and with the number of sampled responses much smaller than the covariates' dimension, intriguing ``all-or-nothing'' phase transitions occur where the recovery error jumps from almost zero to its maximum value at a certain threshold \cite{david2017high,reeves2019all,reeves2019all_2,DBLP:conf/nips/Niles-WeedZ20,DBLP:conf/nips/LuneauBM20,DBLP:conf/nips/BarbierMR20,niles2021all}. For a recent review on the Bayesian approach to high-dimensional inference, see \cite{banerjee2021bayesian}.

\subsection{Organization} The paper is organized as follows. In Section~\ref{sec:2}, we explain the link between the high-dimensional regression problem presented here and the ST spin glass model. 
In Section~\ref{sec:3}, we first state two results on the main quantities of interest in spin glass models, namely, the free energy (i.e., log-partition function) and overlaps (that will later be related to the mean-square error) in the ST model. We then state our main results on the regression task with a particular emphasis on the classical example of the Bayesian regression with quadratic loss and gaussian prior. Sections \ref{sec:4}--\ref{sec:freeenergy} are devoted to establishing results concerning the ST model. More explicitly, Section~\ref{sec:4} establishes a number of concentration properties for the overlaps and free energy. Section~\ref{sec:5} shows how to obtain certain fixed-point equations for the ``order parameters'' of the model based on the concentration results in Section~\ref{sec:4} by utilizing the cavity method and simple gaussian integration by parts. Section~\ref{sec:freeenergy} proves a closed-form expression for the limiting log-partition function of the posterior in a straightforward manner thanks to the convergence results for the order parameters. Finally, our results on the regression task are established in Section~\ref{sec:7}.

\section{Connection to the Shcherbina-Tirozzi  model}\label{sec:2}

In this section, we explain how our regression task can be related to a generalization of the ST model \cite{shcherbina2003rigorous}, a connection mentioned (but not exploited) in \cite{donoho2016high}. For $h\in \mathbb{R},$ consider the \emph{Hamiltonian} (i.e., energy function) $H_N=(H_N(\sigma))_{\sigma\in \mathbb{R}^N}$ defined by
\begin{align}\label{HamHN}
	H_N(\sigma) := \sum_{k\leq M} u\big(S_k(\sigma)+z_k\big) - h \sum_{i\leq N} \sigma_i - \kappa \|\sigma\|^2
\end{align}
and the corresponding log-partition function, also called \emph{free energy}, by
\begin{align}\label{add:eq-2}
	F_N& :=\frac{1}{N}\ln \int_{\mathbb{R}^N}\exp H_N(\sigma)\,d\sigma.
\end{align}
 In the terminology of the large deviation theory \cite{Dembo1998LargeDT,touchette2009large}, it is the scaled cumulant log-moment generating function associated with the energy per variable $H_{N}(\sigma)/N$.
When $z_k=0$ and $u$ grows at most linearly, the Hamiltonian $H_N$ is known as the original ST model  extensively studied in \cite{Talagrand2011spina}. In the setting of the present paper,  we  allow $u$ to grow quadratically (see Assumption \ref{ass:growth_u} below for details).

%This is a variant of the Gardner spin glass model \cite{Gardner} studied by  under a linear growth conditions on $u$, 

For a given realization of the randomness $(G,z)$ defining an instance of the model, we will refer to each element of a sequence $(\sigma^\ell)_{\ell\geq1}$ of independently drawn samples from the \emph{Gibbs measure}
\begin{align*}
    G_N(d\sigma):=\frac{\exp H_N(\sigma)\,d\sigma}{\int \exp H_N(\sigma')\,d\sigma'}
\end{align*}
as \emph{replicas}.

 Denote by $\la\, \cdot\,\ra$ the Gibbs expectation of replicas $(\sigma^\ell)_{\ell\geq 1}$: 
\begin{align}
 \big\langle f((\sigma^\ell)_{\ell\le L})\big\rangle :=\int_{\mathbb{R}^{NL}}f((\sigma^\ell)_{\ell\le L})\prod_{\ell\le L} G_N(d\sigma^\ell) .\label{gibbsExpec}    
\end{align}
Two key quantities for our analysis will be the \emph{overlaps} between replicas defined according to
 \begin{equation}
 	R_{1,1} := \frac{\|\sigma^1\|^{2}}{N}  \quad \mbox{ and }\quad R_{1,2} := \frac{\sigma^1\cdot\sigma^2}{N}  .\label{overlapsR}
 \end{equation}
 Additionally, we also define a pair of ``conjugate overlaps'' in terms of the following two auxiliary quantities
 \begin{align}
 	A(\sigma^\ell) := \Big(u'\big(S_k(\sigma^\ell)+z_k\big)\Big)_{k\le M} \quad \mbox{ and }\quad B(\sigma^\ell) := \Big(u''\big(S_k(\sigma^\ell)+z_k\big)\Big)_{k\le M}.	\label{defs:AB}
 \end{align}

 % ---------------------------------- MODIFICADO HASTA ACA: Principio pagina 5 del documento.

 For convenience, we will sometimes write $A^\ell=(A_k^\ell)_{k\le M}$, $B^\ell=(B_k^\ell)_{k\le M}$, $S_k^\ell$ instead of $A(\sigma^\ell)$, $B(\sigma^\ell)$, or $S_k(\sigma^\ell)$. Whenever the replica index $\ell$ is omitted, it is assumed to be equal to $1$. The conjugate overlaps are then defined by
 \begin{equation}
 	Q_{1,1} := \frac{1}{N} \sum_{k\leq M} \Big((A_k)^2 + B_k\Big) \quad \mbox{ and }\quad Q_{1,2} := \frac{1}{N} \langle A^1,A^2\rangle .\label{overlapsQ}
 \end{equation}
 As one shall see, the vector $(R_{1,1},R_{1,2},Q_{1,1},Q_{1,2})$ of overlaps order parameters is the ``correct'' one, in the sense that all the statistical analysis can be asymptotically  expressed solely in terms of the limits of these quantities and the model at the ``macroscopic level'' is completely determined by them.

In view of our regression task, the posterior measure \eqref{posterior} can be thought of as a Gibbs measure of the spin glass model with Hamiltonian
\begin{align*}
	\tilde H_{N}(x):=\sum_{k\leq M}u\Big(\frac{1}{\sqrt{N}}\sum_{i\leq N}g_{ki}(x_i-x_i^*)-z_k\Big)-\kappa \|x\|^2,
\end{align*}
that is,
\begin{align}
	\label{gibbs}
	P(dx\mid \mathcal{D})=\tilde G_N(dx):=\frac{\exp{\tilde H_N(x)}\,dx}{\int \exp{\tilde H_N(x')}\,dx'}.
\end{align}
As in the ST model, we define the free energy by
$$
\tilde F_N:=\frac{1}{N}\ln \int_{\mathbb{R}^N}\exp \tilde H_N(x)\,dx.
$$
Denote by $(x^\ell)_{\ell\geq 1}$ the replicas drawn i.i.d. from measure \eqref{gibbs} and by $\la\, \cdot\,\ra^\sim$ the Gibbs expectation associated to the Hamiltonian $\tilde H_N$, defined similarly as \eqref{gibbsExpec} but instead with $\tilde G_N$.

In particular, from \eqref{estimator}, we readily see that the Bayesian estimator $$\hat x=\la x\ra^\sim.$$

The connection between the ST and our models relies on  a number of change of variables. First of all, by letting $w:=x-x^*$, we may write
\begin{align*}
%	\begin{split}\label{add:eq-4}
\int_{\mathbb{R}^N} \exp \tilde  H_N(x) \,dx&=\int_{\mathbb{R}^N}\exp \Big(\sum_{k\leq M}u\Big(\frac{1}{\sqrt{N}}\sum_{i\leq N}g_{ki}w_i-z_k\Big)-\kappa \|w+x^*\|^2\Big)\,dw\\
&=\exp(-\kappa\|x^*\|^2)\int_{\mathbb{R}^N}\exp \Big(\sum_{k\leq M}u\Big(\frac{1}{\sqrt{N}}\sum_{i\leq N}g_{ki}w_i-z_k\Big)\\
&\qquad\qquad\qquad\qquad\qquad\qquad\qquad-2\kappa\sum_{i\le N}x_i^*w_i-\kappa\|w\|^2\Big)\,dw.
%\end{split}
\end{align*}
To  simplify this integral, take an orthogonal matrix $O\in \mathcal{O}(N)$ such that $Ox^*=N^{-1/2}{\|x^*\|}\boldsymbol{1}_N$. Note that conditionally on $x^*$,
\begin{align*}
\e (GO^\intercal)_{k\ell }(GO^\intercal)_{k' \ell' }&=\e \Big(\sum_{i\le N}g_{ki}O_{\ell i}\Big)\Big(\sum_{i\le N}g_{k'i}O_{\ell' i}\Big)\\
&=\delta_{kk'}\sum_{i\le N}O_{\ell i}O_{\ell' i}=\delta_{kk'}(OO^\intercal )_{\ell\ell'}=\delta_{kk'}\delta_{\ell\ell'},
\end{align*}
which implies that $GO^\intercal \sim G$ conditionally on $x^*$. Note also that $\langle x^*,O^\intercal a\rangle =N^{-1/2}\|x^*\|\langle \boldsymbol{1}_N,a\rangle$ and $\|O^\intercal a\|=\|a\|.$ From these and using the change of variable $\sigma:=Ow$,
\begin{align}
	\begin{split}\label{add:eq-5.1}
	&\int_{\mathbb{R}^N}\exp \Big(\sum_{k\leq M}u\Big(\frac{1}{\sqrt{N}}\sum_{i\leq N}g_{ki}w_i-z_k\Big)-2\kappa\sum_{i\le N}x_i^*w_i-\kappa \|w\|^2\Big)\,dw\\
	&\qquad=\int_{\mathbb{R}^N}\exp \Big(\sum_{k\leq M}u\Big((\bar G O^\intercal \sigma)_k-z_k\Big)-2\kappa \langle x^*,O^\intercal \sigma\rangle -\kappa \|O^\intercal \sigma\|^2\Big)\,d\sigma\\
	&\qquad=\int_{\mathbb{R}^N}\exp \Big(\sum_{k\leq M}u\Big((\bar G O^\intercal\sigma)_k-z_k\Big)-\frac{2\kappa\|x^*\|}{\sqrt{N}}\sum_{i\le N}\sigma_i-\kappa \|\sigma\|^2\Big)\,d\sigma,
	\end{split}
\end{align}
which implies, from the rotational invariance of $G$ mentioned above and the symmetry of $z_k,$ that
\begin{align*}
%	\begin{split}\label{add:eq-5}
&\e\ln\int_{\mathbb{R}^N}\exp \Big(\sum_{k\leq M}u\Big(\frac{1}{\sqrt{N}}\sum_{i\leq N}g_{ki}w_i-z_k\Big)-2\kappa\sum_{i\le N}x_i^*w_i-\kappa \|w\|^2\Big)\,dw\\
&\qquad=\e\ln\int_{\mathbb{R}^N}\exp \Big(\sum_{k\leq M}u\Big((\bar G \sigma)_k+z_k\Big)-\frac{2\kappa\|x^*\|}{\sqrt{N}}\sum_{i\le N}\sigma_i-\kappa \|\sigma\|^2\Big)\,d\sigma.
%\end{split}
\end{align*}
To sum up,
\begin{align*}
\e\tilde F_N=-\kappa \e\gamma_N+\frac{1}{N}\e\ln \int_{\mathbb{R}^N}\exp \Bigl( \sum_{k\leq M} u\big(S_k(\sigma)+z_k\big) - 2\kappa\sqrt{\gamma_N} \sum_{i\leq N} \sigma_i - \kappa \|\sigma\|^2\Bigr)\,d\sigma
\end{align*}
for $\gamma_N:=\|x^*\|^2/N.$
% From the same reason, if one is interested in the Hamiltonian instead
% \begin{align*}
% H_{N}^h(\sigma)=\sum_{k\leq M}u\Big(\frac{1}{\sqrt{N}}\sum_{i\leq N}g_{ik}(\sigma_i-\sigma_i^*)\Big)+\sum_{i\le N}h(\sigma_i)-\kappa \|\sigma_i\|^2.
% \end{align*}
% for some concave function $h$, then we can again perform a change of variable to get
% \begin{align*}
% &\frac{1}{N}\e \ln\int_{\mathbb{R}^N} \exp H_N^h(\sigma) d\sigma\\
% &\approx-\gamma \kappa+\frac{1}{N}\e\ln \int_{\mathbb{R}^N}\exp \Big(\sum_{k\leq M}u\Big(\frac{1}{\sqrt{N}}\sum_{i\leq N}g_{ik}x_j\Big)+\sum_{i=1}^N\Big(h(x_i+\sigma_i)-\kappa\sqrt{\gamma}x_i\Big)-\kappa\|x\|^2\Big)dx.
% \end{align*}
% Note that to incorporate the noise term, we simply write
% \begin{align*}
% H_{N}^{h,\Delta_*}(\sigma)=\sum_{k\leq M}u\Big(\frac{1}{\sqrt{N}}\sum_{i\leq N}g_{ik}(\sigma_i-\sigma_i^*)+\sqrt{\Delta_*}z_i\Big)+\sum_{i=1}^Nh(\sigma_i)-\kappa \|\sigma_i\|^2
% \end{align*}
% and again,
% \begin{align*}
% &\frac{1}{N}\e \ln\int_{\mathbb{R}^N} \exp H_N^{h,\Delta_*}(\sigma) d\sigma\\
% &\approx-\gamma \kappa+\frac{1}{N}\e\ln \int_{\mathbb{R}^N}\exp \Big(\sum_{k\leq M}u\Big(\frac{1}{\sqrt{N}}\sum_{i\leq N}g_{ik}x_j+\sqrt{\Delta_*}z_i\Big)+\sum_{i=1}^N\Big(h(x_i+\sigma_i)-\kappa\sqrt{\gamma}x_i\Big)-\kappa\|x\|^2\Big)dx.
% \end{align*}
%
If now we randomize $h=h_N=2\kappa\sqrt{\gamma_N}$ in the Hamiltonian \eqref{HamHN} of the ST model independently of all other randomness, then we readily get
\begin{align}\label{add:eq-1}
	\e\tilde F_N=-\kappa\e\gamma_N+\e F_N
\end{align}
and similarly,
\begin{align}\label{add:eq-3}
\e\|\hat x-x^*\|^2=\e \| \langle x\rangle^\sim-x^*\|^2  =\e \|\langle \sigma\rangle\|^2.
\end{align} 
Consequently, to understand our regression problem, we may study the limiting free energy and overlaps corresponding to the generalized ST model first. This characterization is part of our main results and is stated in the next section.

\section{Main results} \label{sec:3}
The first part of our main results consists of an expression for the limiting free energy and the convergence of the overlaps in the generalized ST model defined in \eqref{HamHN}. Throughout the rest of this paper, the following assumption on the function $u$ is in force. 
\begin{assumption}[Growth condition]\label{ass:growth_u}
	The function $u(s)$ is concave, non-positive, and there exists a constant $d>0$ such that 
	\begin{align*}
	%	\label{G1intro}
		\max\bigl(|u(0)|, |u'(0)|, \|u''\|_\infty, \|u'''\|_\infty, \|u^{(4)}\|_\infty \bigr)\leq d
	\end{align*}
	and
	\begin{align*}
		%\label{G3intro}
		|u'(s)|\leq d(1+\sqrt{|u(s)|}).
	\end{align*}
\end{assumption}

\begin{remark}
	\rm The concavity of $u$ is conceptually an important assumption, which ensures that our posterior distribution is a log-concave measure and, as a result, the overlaps are concentrated under the Gibbs measure $G_N$. While the other assumptions are purely technical,  overall they weaken the settings considered in \cite{Talagrand2011spina}. In particular, in this reference $u$ is allowed to grow at most linearly and thus its results do not apply to our main example, namely, a quadratic $u$.
\end{remark}

To formulate our limiting free energy, we introduce $\bar F:\mathbb{R}^{4}\mapsto \mathbb{R}$ such that
\begin{align}
  \bar F(q,\rho,r,\bar r)&:=\alpha\e \ln \e_{\xi} \exp u\big(\tilde z\sqrt{\Delta_*+q} +\xi \sqrt{\rho-q}\big)
  \nonumber
  \\ &\qquad  +\frac12\Big(\frac{r+h^2}{2\kappa + r-\bar r} - \ln(2\kappa +r-\bar r) +rq-\bar r \rho+ \ln 2\pi\Big)
\label{barFdef} 
\end{align}
for $0\leq q<\rho$ and $0\leq \bar r \leq r$, where $\tilde z,\xi$ are independent random variables with normal law $\mathcal{N}(0,1)$. If we let $\theta= \tilde z\sqrt{\Delta_*+q} +\xi \sqrt{\rho-q}$, basic gaussian integration by parts with respect to $\tilde z$ and $\xi$ implies that the critical points of the function $\bar F$ must satisfy the following system of equations:
\begin{align}
    q &= \Psi(r,\bar r) := \frac{r+h^2}{(2\kappa+r-\bar r)^2},
    \label{CPeqQ}
    \\ 
    \rho&= \bar \Psi(r,\bar r) := \frac{1}{2\kappa+r-\bar r}+ \frac{r+h^2}{(2\kappa+r-\bar r)^2},
    \label{CPeqRho}
    \\ 
    r&= \Phi(q,\rho):= \alpha \e \Bigl(\frac{\e_\xi u'(\theta)\exp{u(\theta)}}{\e_\xi \exp{u(\theta)}}\Bigr)^2,
    \label{CPeqR}
    \\ 
    \bar r&= \mybar \Phi(q,\rho):= \alpha \e \frac{\e_\xi (u''(\theta)+u'(\theta)^2)\exp{u(\theta)}}{\e_\xi \exp{u(\theta)}}.
    \label{CPeqRb}
\end{align}
In addition, for any $0\leq q<\rho,$ we consider the function $F:\mathbb{R}^{2}\mapsto \mathbb{R}$ defined as
\begin{align}
	F(q,\rho)&:=\alpha\e \ln \e_{\xi} \exp u\big(\tilde z\sqrt{\Delta_*+q} +\xi \sqrt{\rho-q}\big)
	\nonumber\\
	&\qquad + \frac{1}{2}\Big(\ln(\rho-q) +h^2 (\rho-q) +\frac{\rho}{\rho-q} - 2\kappa\rho + \ln2\pi\Big).
	\label{Fdef}  
\end{align}
Observe that if $(q,\rho,r,\bar r)$ satisfies \eqref{CPeqQ}--\eqref{CPeqRb}, then   $\bar F(q,\rho,r,\bar r)$ and $F(q,\rho)$ match each other  as can be seen by  rewriting \eqref{CPeqQ} and \eqref{CPeqRho} as 
$$
2\kappa + r-\bar r = \frac{1}{\rho-q} \quad \mbox{and} \quad r+h^2 = \frac{q}{(\rho-q)^2}
$$
and plugging these into the second line in \eqref{barFdef} after writing $rq-\bar r \rho = -r(\rho-q)+(r-\bar r)\rho$. 
Here, if we express $r,\bar r$ in terms of $q,\rho$ in (\ref{CPeqR})--(\ref{CPeqRb}), the pair $(q,\rho)$ is also a critical point of $F$. Our first main result states that the limiting free energy in the ST model is equal to $F(q,\rho)$ as long as \eqref{CPeqQ}--\eqref{CPeqRb} have a unique solution.

	\begin{theorem}[Free energy of the ST spin glass]\label{ThmMain}
		For any $\kappa>0$, $h\in\Reals$, and $\Delta_*\geq 0$, if (\ref{CPeqQ})--(\ref{CPeqRb}) have a unique solution $(q,\rho,r,\bar r)$, then $$\lim_{N\to\infty}\e\bigl|F_N-F(q,\rho)\bigr|^2=0.$$
	\end{theorem}

The next result  establishes the convergence of the overlaps \eqref{overlapsR} and \eqref{overlapsQ}.

\begin{theorem}[Convergence of overlaps]\label{thm:conv_to_solution}
    For any $\kappa>0$, $h\in\Reals$, and $\Delta_*\geq 0$, if (\ref{CPeqQ})--(\ref{CPeqRb}) have a unique solution $(q,\rho,r,\bar r)$, then under $\e\la\,\cdot\,\ra$ and as $N$ diverges,
    \begin{equation*}
        R_{1,1} \xrightarrow{L^2} \rho, \quad R_{1,2}\xrightarrow{L^2}q, \quad Q_{1,1} \xrightarrow{L^2} \bar r, \quad \mbox{and} \quad Q_{1,2}\xrightarrow{L^2}r.
    \end{equation*}
\end{theorem}

From these and with the help of the identities \eqref{add:eq-1} and \eqref{add:eq-3}, we can now compute the limiting free energy and the mean-square error associated to our regression task.

\begin{theorem}[Free energy and mean-square error in mismatched regression]\label{add:thm1}
	Assume the convergence \eqref{convNorm} for the norm of the ground-truth coefficients towards some deterministic $\sqrt\gamma\geq 0$. For $\kappa>0$, $h=2\kappa \sqrt{\gamma}$, and $\Delta_*\geq 0$, if (\ref{CPeqQ})--(\ref{CPeqRb}) have a unique solution $(q,\rho,r,\bar r)$, then
	\begin{align}\label{add:thm1:eq1}
		\lim_{N\to\infty}\e\bigl|\tilde F_N-(-\kappa \gamma+F(q,\rho))\bigr|=0
	\end{align}
and
\begin{align}\label{add:thm1:eq2}
\lim_{N\to\infty}\frac{1}{N}\e\|\hat x-x^*\|^2=\lim_{N\to\infty} \frac1N \e \| \langle x\rangle^\sim-x^*\|^2=q.
\end{align}

\end{theorem}

\begin{example}[Quadratic loss]\rm
Let $u(s)= -{s^2}/(2\Delta),$ where $\Delta>0$ is not necessarily equal to the true gaussian noise variance $\Delta_*.$
Obviously $u$ satisfies Assumption \ref{ass:growth_u}. More importantly, it can be checked that \eqref{CPeqQ}--\eqref{CPeqRb} have a unique solution given by 
	\begin{equation}\label{eq:vals_overlaps}
		\begin{cases}
			q = \frac{c^2(h^2(\Delta+c)^2+\Delta_*\alpha)}{(\Delta+c)^2-\alpha c^2} , \\
			\rho = q + c,\\
			r= \frac{\alpha(\Delta_*+q)}{(\Delta+\rho-q)^2},\\
			\bar r= r- \frac{\alpha}{\Delta+\rho-q},
		\end{cases} 
	\end{equation}
	where the constant
	\begin{align*}
		c := \frac{\sqrt{(2\kappa\Delta+\alpha-1)^2 + 8 \kappa\Delta}-(2\kappa\Delta+\alpha-1)}{4\kappa} .  
	\end{align*}
	To see this, note that (\ref{CPeqR}) and (\ref{CPeqRb}) can be written as
	\begin{align}
		r&= \frac{\alpha(\Delta_*+q)}{(\Delta+\rho-q)^2} \quad \mbox{and} \quad
		\bar r= r- \frac{\alpha}{\Delta+\rho-q}.
		\label{CPeqRbquad}
	\end{align}
	Take $a = \rho - q$ and $a' = r- \bar r$. From the difference between \eqref{CPeqQ} and \eqref{CPeqRho} and that of \eqref{CPeqR} and \eqref{CPeqRb}, we arrive at
	\begin{equation}\label{eq:eqs_aap}
		\begin{cases}
			a a' + 2 \kappa a = 1, \\
			a a' + \Delta a' = \alpha,
		\end{cases}
	\end{equation}
	from which we can eliminate $a'$ to get a single equation for $a$,
	\begin{equation*}
		2 \kappa a^2 + \left(2\kappa \Delta + \alpha -1 \right) a - \Delta = 0.
	\end{equation*}
	It is easy to see that the only non-negative solution to this equation is $a = c$ and thus,
	$
	a' = \alpha/(\Delta + c).
	$
	Finally, using these values of $a$ and $a'$ and the second equation of \eqref{CPeqRbquad} leads to
	\begin{equation}
		q = \frac{c^2(h^2(\Delta+c)^2+\Delta_*\alpha)}{(\Delta+c)^2-\alpha c^2}. \label{mmse_ridge}
	\end{equation}
	Note that $q$ is positive since \eqref{eq:eqs_aap} implies that $aa'<1$ and $aa'<\alpha$, which ensures that $(\Delta+c)^2/c^2=(a a')^2 < \alpha$. 
	From this $q$, $(\rho,r,\bar r)$ can then be uniquely determined through the second to the fourth equations in \eqref{eq:vals_overlaps}.
	
	Surprisingly, if we parametrize the posterior distribution by an ``inverse temperature'' hyperparameter $\beta>0$, namely, the original Gibbs measure $\tilde G_{N}$ in \eqref{gibbs} becomes
	\begin{align*}
		\tilde G_{N,\beta}(dx)=\frac{\exp (\beta \tilde H_N(x))\,dx}{\int \exp(\beta \tilde H_N(x'))\,dx'},
	\end{align*}
meaning that $\Delta$ and $\kappa$ are replaced by $\Delta/\beta$ and $\beta \kappa$ in the original model, then the resulting constant $q_\beta$, i.e., the mean-square error, remains the same as the original $q$ for any inverse temperature $\beta>0$. This can be easily seen from the above equations (for quadratic $u$). At the same time, formula  \eqref{mmse_ridge} matches the one for ridge regression at zero temperature (see, e.g., equation (26) in \cite{thrampoulidis2018precise}). This non-trivial fact is due to the rotational invariance of the model following from the isotropy of the covariates. But as mentioned already in the introduction, this is not a priori evident in the high-dimensional regime we consider here, where $M$ and $N$ are comparable. It would be interesting to explore in future research for which settings posterior mean and mode (i.e., sampling and optimizing) yield the same reconstruction performance.
\end{example}

\section{Concentration} \label{sec:4}
In this section we will prove a number of useful concentration results for overlaps and the free energy. For this, it will be convenient to restate the growth conditions in Assumption \ref{ass:growth_u} in a more general way, more suitable for the proof. Take $n\geq 1$ and consider two functions $\phi,\psi : \Reals^n\mapsto\Reals$ which together with the function $u : \Reals\mapsto\Reals$ satisfy the following set of conditions. Namely, we suppose that there exist $d>0$ and $\eta_0>0$ such that the following holds:

\begin{enumerate}
    \item[(G-1)] The function $u(s)$ satisfies
    \begin{align}\label{G1gen}
        0\geq u(s)\geq -d(1+s^2) \quad \mbox{and} \quad |u'(s)|\leq d(1+|s|).
    \end{align}

    \item[(G-2)] The functions $\phi,\psi : \mathbb{R}^n\mapsto \mathbb{R}$ satisfy
  \begin{align}\label{lip}
      \max\bigl(|\phi(s)-\phi(s')|,|\psi(s)-\psi(s')|\bigr)
      &\leq d\bigl(1+\|s\|+\|s'\|\bigr)\|s-s'\|,
        \\
      \label{phipsi}
      \max\bigl(|\phi(s)|,|\psi(s)|\bigr)&\leq d\bigl(1+\|s\|^2\bigr),
        \\
        \label{bound}
      \max_{1\leq \ell,\ell'\leq n}\bigl(|\partial_{s_\ell s_{\ell'}}\phi(s)|,|\partial_{s_\ell s_{\ell'}}\psi(s)|\bigr)&\leq d.
  \end{align}
  
    \item[(G-3)] The following inequality holds:
  \begin{align*}
  	%\label{bound2}
    \sup_{y\in \mathbb{R}^n}\Bigl(\sum_{\ell \le n} u(y^\ell)+\eta_0 |\psi(y^1,\ldots,y^n)|\Bigr)\leq \alpha d.
  \end{align*}
\end{enumerate} 
The last assumption (G-3) will be used in an equivalent form,
\begin{align}\label{bound2eq}
    \max_{y^1,\ldots,y^n\in \mathbb{R}^M}\sum_{k\le M}\Bigl(\sum_{\ell \le n} u(y_k^\ell)+\eta_0 |\psi(y_k^1,\ldots,y_k^n)|\Bigr)\leq N d.
\end{align} 
Also, the assumptions (\ref{lip}), (\ref{phipsi}) may look redundant given (\ref{bound}), but we state them for convenience.

We will later specialize these assumptions for the following two specific choices:
    \begin{enumerate}
        \item[$\bullet$] For $n=1$: $\phi(s)=s^2$ and $\psi(s)=u''(s)+u'(s)^2$, 
        \item[$\bullet$] For $n=2$:  $\phi(s_1,s_2)=s_1s_2$ and $\psi(s_1,s_2)=u'(s_1)u'(s_2)$.
    \end{enumerate} 
    One can easily see that Assumption \ref{ass:growth_u} on $u$ implies (G-1)--(G-3) in both cases, by possibly increasing the value of $d$.

For $\sigma^1,\ldots,\sigma^n\in \mathbb{R}^N$ denote $\vec\sigma=(\sigma^1,\ldots,\sigma^n)\in \mathbb{R}^{Nn}$ and set $\|\vec\sigma\|=(\|\sigma^1\|^2+\cdots+\|\sigma^n\|^2)^{1/2}$. Consider any $\lambda$ and $\eta$ satisfying $|\lambda|,|\eta|\leq \min(\eta_0,\kappa/(2d))$.    
For any $\sigma\in \mathbb{R}^N$ recall $S_k(\sigma):=(\bar G\sigma)_k$ and define
\begin{align*}
  X_N(\sigma):=\sum_{k\le M}u\bigl(S_k(\sigma)+z_k\bigr).
\end{align*}
For $\vec\sigma\in \mathbb{R}^{Nn}$, we define an Hamiltonian between ``coupled replicas'' $(\sigma^\ell)_{l \leq n}$:
\begin{align}\label{HamXLE}
  X_{N,\lambda,\eta}(\vec\sigma):=\sum_{\ell\le n} X_N(\sigma^\ell) + h\sum_{\ell\le n}\sum_{i\le N}\sigma_i^\ell-\kappa \|\vec\sigma\|^2+\lambda NR(\vec\sigma)+\eta NQ(\vec\sigma),
\end{align}
where the generalized overlaps are
\begin{align}
  R(\vec\sigma)&:=\frac{1}{N}\sum_{i\le N}\phi(\sigma_i^1,\ldots,\sigma_i^n),\label{gene_overlap_R}\\
  Q(\vec\sigma)&:=\frac{1}{N}\sum_{k\le M}\psi\bigl(S_k(\sigma^1)+z_k,\ldots,S_k(\sigma^n)+z_k\bigr),\label{gene_overlap_Q}
\end{align}
for functions $\phi$ and $\psi$ verifying assumptions (G-1)--(G-3). Define the (non-averaged) free energy associated to $X_{N,\lambda,\eta}$ by
\begin{align*}
  F_N(\lambda,\eta)&:=\frac{1}{N}\ln Z_{N,\lambda,\eta}
    \quad\mbox{where}\quad
  Z_{N,\lambda,\eta}:=\int_{\mathbb{R}^{Nn}}\exp{X_{N,\lambda,\eta}(\vec\sigma)}\,d\vec\sigma.
\end{align*}
%\col{The expectation of this generalized log-partition function (which concentrates as shown later in Lemma~\ref{lem:conc_fe}) is the log-moment generating %function from which the large deviations of the macroscopic observables of the system can be extracted \cite{Dembo1998LargeDT,touchette2009large}.}
Note that assumptions \eqref{G1gen}, \eqref{lip}, and \eqref{bound2eq} as well as $|\lambda|,|\eta|\leq \min(\eta_0,\kappa/(2d))$ ensure that this quantity is a.s. finite in $G$ and $z=(z_k)_{k\le M}$
  Denote by $\la \,\cdot\,\ra_{\lambda,\eta}$ the expectation with respect to the Gibbs measure proportional to $\exp X_{N,\lambda,\eta}(\vec\sigma)$:
\begin{align*}
 \la g(\vec\sigma)\ra_{\lambda,\eta} :=\frac 1{Z_{N,\lambda,\eta}}\int_{\mathbb{R}^{Nn}} g(\vec\sigma) \exp{X_{N,\lambda,\eta}(\vec\sigma)}\,d\vec\sigma.
\end{align*}

From now on, when we say that ``there exist constants $\varepsilon, K,K', D,L$, etc.'', these constants are functions of all the parameters of the model $(d, n, \alpha, \Delta_*, \kappa, \eta_0,h)$, which are bounded uniformly over compacts. As before, $\mathbb{E}$ is the expectation with respect to the gaussians $(G,z)$.

\begin{proposition}\label{ThmOConc}
  Under the above assumptions, there exists some $K>0$ such that 
  \begin{align}
    \e \Bigl\la \Bigl|R(\vec\sigma)-\e\la R(\vec\sigma)\ra_{0,0}\Bigr|^2\Bigr\ra_{0,0}
    &\leq \frac{K}{\sqrt{N}} \quad \mbox{and} \quad \e \Bigl\la \Bigl|Q(\vec\sigma)-\e\la Q(\vec\sigma)\ra_{0,0}\Bigr|^2\Bigr\ra_{0,0}
    \leq \frac{K}{\sqrt{N}}.
    \label{ThMRcond1}
  \end{align}
  % and
  % \begin{align}
  %   \e \Bigl\la \Bigl|Q(\vec\sigma)-\e\la Q(\vec\sigma)\ra_{0,0}\Bigr|^2\Bigr\ra_{0,0}
  %   &\leq \frac{K}{\sqrt{N}}.
  %   \label{ThMRcond2}
  % \end{align}
    Moreover, for any $k\geq 1$, there exists  some $C_k>0$ such that 
  \begin{align*}
 %     \label{ThMRcond3:eq1}
        \e\bigl\la |R(\vec\sigma)|^k\bigr\ra_{0,0}\leq C_k \quad \mbox{ and }\quad   
        \e\bigl\la |Q(\vec\sigma)|^k\bigr\ra_{0,0} \leq C_k.
  \end{align*}   
\end{proposition}

In the rest of this section, we will establish Proposition~\ref{ThmOConc}. To fix our notation, for any real $M\times N$-matrix $A$, define the Frobenius and operator norms, respectively, as
\begin{align*}
    \|A\|:=\Bigl(\sum_{k\le M}\sum_{i\le N}A_{ki}^2\Bigr)^{1/2} \quad \mbox{and} \quad \|A\|_{op}:=\sup_{y\in \mathbb{R}^N:\|y\|\leq 1}\|Ay\|.
\end{align*}
Recall $\| \cdot\|$ applied to a vector is the $L_2$ norm and note that $\|A\|_{op}\leq \|A\|$.

% ----------------------------------------------------------------------------

% {For certain results proven below we explicitly keep track of the dependency in $h$ considered as a random variable, despite it is at the moment a constant, as it will later be useful when we will randomize it}.

\subsection{Control of the self-overlap}

Let $\bar z:=z/\sqrt{N}$ and recall also that $\bar G:=G/\sqrt{N}$.

\begin{lemma}\label{lem1}
     {There exist $\varepsilon>0$ and $K\geq 1$ such that if 
    $
        t,|\lambda|,|\eta|\leq \varepsilon
    $}
    then
  \begin{align*}
      \bigl\la \exp{ t\|\vec\sigma\|^2}\bigr\ra_{\lambda,\eta}\leq \exp{KN(\|\bar z\|^2+\|\bar G\|_{op}^2+1+h^2)}.
  \end{align*}      
\end{lemma}
\noindent
    \emph{Proof.}
    From \eqref{phipsi}, 
    \begin{align}\label{eq6}
        |R(\vec{\sigma})|&\leq \frac{d}{N}\sum_{i\le N}\Bigl(1+\sum_{\ell \le n}(\sigma_i^\ell)^2\Bigr)= \frac{d}{N}(N+\|\vec\sigma\|^2)
    \end{align}
    and
    \begin{align}\label{eq8}
        |Q(\vec\sigma)|&\leq \frac{d}{N}\Bigl(M+\sum_{\ell \le n}\sum_{k\le M}(S_k(\sigma^\ell)+z_k)^2\Bigr)
        \leq  \frac{d}{N}\Bigl(M+2\|\bar G\|_{op}^2\|\vec\sigma\|^2+2n\|z\|^2\Bigr).
    \end{align}
    By \eqref{eq6}, if $t+d|\lambda|\leq \kappa/2$, then  
    \begin{align*}
        &t\|\vec\sigma\|^2-\kappa \|\vec\sigma\|^2+\lambda NR(\vec\sigma)\leq
        |\lambda| d N -\frac{\kappa}{2}\|\vec\sigma\|^2,
    \end{align*}
    which, together with \eqref{bound2eq} for $|\eta|\leq \eta_0$ implies that, 
    \begin{align}
        \int \exp\big({t\|\vec\sigma\|^2+X_{N,\lambda,\eta}(\vec\sigma)}\big)\,d\vec\sigma&
        \leq  \exp N\Bigl(d+d|\lambda|+\frac{nh^2}{2\kappa}+\frac{n}{2}\ln \frac{2\pi}{\kappa}\Bigr),
        \label{eq7}
    \end{align}
    where we used that for any $a>0$ and $b\in \mathbb{R}$,
    \begin{align}\label{formula}
        \int \exp({-ax^2+bx})\,dx&=\Bigl(\frac{\pi}{a}\Bigr)^{1/2}\exp\frac{b^2}{4a}.
    \end{align}
    Next, since by \eqref{G1gen}
    \begin{align*}
        u(S_k(\sigma)+z_k)&\geq -d(1+2S_k(\sigma)^2+2z_k^2),
    \end{align*}
    we have
    \begin{align*}
        \sum_{k\le M}u(S_k(\sigma)+z_k)&\geq -d\bigl(M+2\|\bar G\|_{op}^2\|\sigma\|^2+2\|z\|^2\bigr).
    \end{align*}
    Combining this with \eqref{eq6} and \eqref{eq8}, we obtain that
    \begin{align*}
        X_{N,\lambda,\eta}(\vec\sigma)&\geq -\tilde{T} \|\vec\sigma\|^2-N\tilde{T}'+h\sum_{\ell \le n}\sum_{i\le N}\sigma_i^\ell,
    \end{align*}
    where 
    \begin{align*}
        \tilde{T}&:=\kappa+{2d(1+|\eta|)}\|\bar G\|_{op}^2+d|\lambda| > 0 \quad \mbox{and}\quad \tilde{T}':=d\bigl(\alpha n +|\lambda|+\alpha |\eta|\bigr)+2nd(1+|\eta|)\|\bar z\|^2 >0.
    \end{align*}
    Consequently, from \eqref{formula},
    \begin{align*}
        Z_{N,\lambda,\eta} \geq \exp nN\Bigl(-\frac{\tilde T'}{n}+\frac{h^2}{4\tilde T}+\frac{1}{2}\ln \frac{\pi}{\tilde T}\Bigr) \geq \exp nN\Bigl(-\frac{\tilde T'}{n}-\frac{1}{2}\ln {\tilde T}\Bigr),
    \end{align*}
    which together with the fact that $\ln x\leq x+1$ for all $x>0$ implies that
    \begin{align*}
        \frac{1}{Z_{N,\lambda,\eta}} \leq \exp nN\Bigl(\frac{\tilde T'}{n}+\frac{1}{2}({\tilde T}+1)\Bigr)\leq \exp K'N\bigl(\|\bar z\|^2+\|\bar G\|_{op}^2+1\bigr).
    \end{align*}
    Our assertion follows by putting this and \eqref{eq7} together. 
    \begin{flushright}$\square$\end{flushright}

{We let $I(A)$ be the indicator function for event $A$.}
\begin{lemma}\label{lem1.5}
  {There exist constants $\varepsilon>0$ and $K,K'>0$ such that, for any $|\lambda|,|\eta|\leq \eps$},
  \begin{align*}
     % \label{lem1.5:eq1}
      \Bigl\la I\Bigl(\|\vec\sigma\|\geq K'\sqrt{N(\|\bar z\|^2+\|\bar G\|_{op}^2+1+h^2)}\Bigr)\Bigr\ra_{\lambda,\eta}&\leq \exp({-KN}).
  \end{align*}
\end{lemma}
\noindent
    \emph{Proof.} Recall the constants $\eps$ and $K$ from the statement of Lemma \ref{lem1}. {If $t\le \varepsilon$,}
  then for any $a>0$, 
  $$
      \bigl\la I(\|\vec\sigma\|\geq a)\bigr\ra_{\lambda,\eta} \leq \bigl\la \exp t(\|\vec\sigma\|^2-a^2)\bigr\ra_{\lambda,\eta}\leq \exp(-{a^2t}+{KN(\|\bar z\|^2+\|\bar G\|_{op}^2+1+h^2)}).
  $$
  Our proof is completed by taking $a=\sqrt{{2t^{-1}KN(\|\bar z\|^2+\|\bar G\|_{op}^2+1+h^2)}}.$ \begin{flushright}$\square$\end{flushright}

\begin{lemma}\label{lem2}
    Let $k\geq 1.$ There exist $0<\varepsilon<1$ and $K\geq 1$ such that, for any $|\lambda|,|\eta|\leq \eps$,
    \begin{align}
        \label{lem2:eq1}
      \bigl\la \|\vec\sigma\|^{2k}\bigr\ra_{\lambda,\eta}&\leq K N^k\bigl(\|\bar z\|^{2k}+\|\bar G\|_{op}^{2k}+1+h^{2k}\bigr).
    \end{align}
\end{lemma}
\noindent
    \emph{Proof.} From Lemma \ref{lem1}, 
  \begin{align*}
      \ln \bigl\la \exp{ t\|\vec\sigma\|^2}\bigr\ra_{\lambda,\eta}\leq {KN(\|\bar z\|^2+\|\bar G\|_{op}^2+1+h^2)}.
  \end{align*}
    For any non-negative random variable $X$, by Lemma 3.1.8 in \cite{Talagrand2011spina},
      $$\e X^k\leq 2^k\bigl(k^k+\bigl(\ln \e \exp X\bigr)^{k}\bigr).$$
  Using this inequality with $X=t\|\vec\sigma\|^{2}$, we have
  \begin{align*}
      t^k\bigl\la \|\vec\sigma\|^{2k}\bigr\ra_{\lambda,\eta}&\leq 2^k\bigl(k^k+K^kN^k(\|\bar z\|^2+\|\bar G\|_{op}^2+1+h^2)^k\bigr)\\
      &\leq 2^k\bigl(k^k+4^kK^kN^k(\|\bar z\|^{2k}+\|\bar G\|_{op}^{2k}+1+h^{2k})\bigr)\\
      &\leq C_kN^k\bigl(\|\bar z\|^{2k}+\|\bar G\|_{op}^{2k}+1+h^{2k}\bigr).
  \end{align*}
    This finishes the proof. \begin{flushright}$\square$\end{flushright}

% ----------------------------------------------------------------------------

\subsection{Convexity and pseudo-Lipschitz property of the overlaps}
\begin{lemma}\label{lem:convex}
    If $|\lambda|+\|\bar G\|_{op}^2 |\eta|\leq \kappa/(d n)$ then the following function is convex:
    \begin{align}\label{lem:convex:eq1}
      \vec\sigma\mapsto\frac{\kappa}{2}\|\vec\sigma\|^2-\lambda NR(\vec\sigma)-\eta NQ(\vec\sigma).
    \end{align}
\end{lemma}
\noindent
    \emph{Proof.} Recall the generalized overlap definitions \eqref{gene_overlap_R} and \eqref{gene_overlap_Q}. Note that for any $\vec y=(y^1,\ldots,y^n)\in \mathbb{R}^N\times\cdots \times\mathbb{R}^N$,
  \begin{align*}
      &\sum_{\ell,\ell'\le n}\sum_{i,i'\le N}\partial_{\sigma_i^\ell \sigma_{i'}^{\ell'}}\Bigl(\sum_{j\le N}\phi(\sigma_j^1,\ldots,\sigma_j^n)\Bigr)y_i^\ell y_{i'}^{\ell'}= \sum_{\ell,\ell'\le n}\Bigl(\sum_{i\le N}\partial_{s_\ell s_{\ell'}}\phi(\sigma_i^1,\ldots,\sigma_i^n) y_i^\ell y_i^{\ell'}\Bigr).
  \end{align*}
  Since, by \eqref{bound}, $|\partial_{s_\ell s_{\ell'}}\phi(s)|\leq d$, this is bounded in absolute value by
  \begin{align*}
        d \sum_{i\le N} \Bigl(\sum_{\ell \le n} |y_i^\ell|\Bigr)^2
        \leq d n \|\vec y\|^2.
  \end{align*}  
    Similarly,
    \begin{align*}
        &\sum_{\ell,\ell'\le n}\sum_{i,i'\le N}\partial_{\sigma_i^\ell\sigma_{i'}^{\ell'}}\Bigl(\sum_{k\le M}\psi\bigl(S_k(\sigma^1)+z_k,\ldots,S_k(\sigma^n)+z_k\bigr)\Bigr)y_i^\ell y_{i'}^{\ell'}\\
        &\qquad=\sum_{\ell,\ell'\le n}\sum_{k\le M}\partial_{s_\ell s_{\ell'}}\psi\bigl(S_k(\sigma^1)+z_k,\ldots,S_k(\sigma^n)+z_k\bigr)\Bigl(\sum_{i\le N} \bar g_{ki}y_i^\ell\Bigr)\Bigl(\sum_{i'\le N} \bar g_{ki'}y_{i'}^{\ell'}\Bigr),
    \end{align*}
    and since, by \eqref{bound}, $|\partial_{s_\ell s_{\ell'}}\psi(s)|\leq d$, this is bounded in absolute value by
  \begin{align*}
        d \sum_{\ell,\ell'=1}^n |(\bar G y^\ell, \bar G y^{\ell'})|
        \leq d\|\bar G\|_{op}^2 \Bigl(\sum_{\ell \le n} \|y^\ell\|\Bigr)^2
        \leq d {n}\|\bar G\|_{op}^2\|\vec y\|^2.
  \end{align*}
  These complete our proof. \begin{flushright}$\square$\end{flushright}

\begin{lemma}\label{lem:lip}
  We have that
  \begin{align*}
    \bigl|R(\vec\sigma)-R(\vec\sigma')\bigr|&\leq \frac{d}{N}\Bigl(\sqrt{N}+\|\vec\sigma\|+\|\vec\sigma'\|\Bigr)\|\vec\sigma-\vec\sigma'\|
  \end{align*}
  and
  \begin{align*}
      &\bigl|Q(\vec\sigma)-Q(\vec\sigma')\bigr|\leq \frac{d}{N}\Bigl(\sqrt{M}+2\sqrt{n}\|z\|+\|\bar G\|_{op}\|\vec\sigma\|+\|\bar G\|_{op}\|\vec\sigma'\|\Bigr) \|\bar G\|_{op} \|\vec\sigma-\vec\sigma'\|.
  \end{align*}
\end{lemma}
\noindent
    \emph{Proof. } For simplicity, for $\vec\sigma=(\sigma^1,\ldots,\sigma^n)$, we denote $\sigma(i)=(\sigma_i^1,\ldots,\sigma_i^n).$ From \eqref{lip}, 
  \begin{align*}
  \bigl|R(\vec\sigma)-R(\vec\sigma')\bigr|&\leq \frac{d}{N}\sum_{i\le N}\bigl(1+\|\sigma(i)\|+\|\sigma'(i)\|\bigr)\|\sigma(i)-\sigma'(i)\|\\
  &\leq \frac{d}{N}\Bigl(\sum_{i\le N}\bigl(1+\|\sigma(i)\|+\|\sigma'(i)\|\bigr)^2\Bigr)^{1/2}\Bigl(\sum_{i\le N}\|\sigma(i)-\sigma'(i)\|^2\Bigr)^{1/2}\\
  &\leq \frac{d}{N}\Bigl(\Bigl(\sum_{i\le N}1\Bigr)^{1/2}+\Bigl(\sum_{i\le N}\|\sigma(i)\|^2\Bigr)^{1/2}+\Bigl(\sum_{i\le N}\|\sigma'(i)\|^2\Bigr)^{1/2}\Bigr)\|\vec{\sigma}-\vec\sigma'\|\\
  &\leq \frac{d}{N}\Bigl(\sqrt{N}+\|\vec\sigma\|+\|\vec\sigma'\|\Bigr)\|\vec\sigma-\vec\sigma'\|,
  \end{align*}
  where in the third line we used the Minkowski inequality.
    Again, by \eqref{lip}, 
    \begin{align*}
        &\bigl|Q(\vec\sigma)-Q(\vec\sigma')\bigr|\\
        &\quad\leq \frac{d}{N}\sum_{k\le M}\Bigl(1+\Bigl(\sum_{\ell \le n}|S_k(\sigma^\ell)+z_k|^2\Bigr)^{1/2}+\Bigl(\sum_{\ell \le n}|S_k({\sigma'}^\ell)+z_k|^2\Bigr)^{1/2}\Bigr)\\
        &\qquad\qquad\qquad\qquad\times\Bigl(\sum_{\ell \le n}S_k(\sigma^\ell-{\sigma'}^{\ell})^2\Bigr)^{1/2}\\
        &\quad\leq \frac{d}{N}\Bigl(\sum_{k\le M}\Bigl(1+\Bigl(\sum_{\ell \le n}|S_k(\sigma^\ell)+z_k|^2\Bigr)^{1/2}+\Bigl(\sum_{\ell \le n}|S_k({\sigma'}^\ell)+z_k|^2\Bigr)^{1/2}\Bigr)^2\Bigr)^{1/2}\\
        &\qquad\qquad\qquad\qquad\times \Bigl(\sum_{k\le M}\sum_{\ell \le n}S_k(\sigma^\ell-{\sigma'}^{\ell})^2\Bigr)^{1/2}.
    \end{align*}
    Here, since 
    $$
        \Bigl(\sum_{k\le M}\sum_{\ell \le n}S_k(\sigma^\ell-{\sigma'}^{\ell})^2\Bigr)^{1/2}
        =
        \Bigl(\sum_{\ell \le n} \|\bar G(\sigma^\ell-{\sigma'}^{\ell})\|^2\Bigr)^{1/2}
        \leq \|\bar G\|_{op}\|\vec\sigma-\vec\sigma'\|,
    $$
    and, similarly,
    \begin{align*}
        \Bigl(\sum_{k\le M}\sum_{\ell \le n}|S_k(\sigma^\ell)+z_k|^2\Bigr)^{1/2}&\leq \Bigl(\sum_{k\le M}\sum_{\ell \le n}S_k(\sigma^\ell)^2\Bigr)^{1/2}+\Bigl(\sum_{k\le M}\sum_{\ell \le n}z_k^2\Bigr)^{1/2} \leq \|\bar G\|_{op}\|\vec\sigma\|+\sqrt{n}\|z\|,
    \end{align*} 
    we can again use the Minkowski inequality to obtain
    \begin{align*}
        &\bigl|Q(\vec\sigma)-Q(\vec\sigma')\bigr|\leq \frac{d}{N}\Bigl(\sqrt{M}+2\sqrt{n}\|z\|+\|\bar G\|_{op}\|\vec\sigma\|+\|\bar G\|_{op}\|\vec\sigma'\|\Bigr) \|\bar G\|_{op}\|\vec\sigma-\vec\sigma'\|.
    \end{align*}
    \begin{flushright}$\square$\end{flushright}

\begin{lemma}\label{lem3}
  For any $k\geq 1$ there exists some $C_k>0$ such that 
  \begin{align*}
    %  \label{lem3:eq1}
        \e\Bigl\la |R(\vec\sigma)|^k\Bigr\ra_{0,0}\leq C_k(1+h^{2k}) \quad \mbox{ and }\quad    
        \e\Bigl\la |Q(\vec\sigma)|^k\Bigr\ra_{0,0} \leq C_k(1+ h^{2k}).
  \end{align*}
\end{lemma}
\noindent
    \emph{Proof.} Let us denote by $\vec s:= \vec\sigma/\sqrt{N}.$ Note that from Lemma \ref{lem:lip},
  \begin{align*}
        \bigl|R(\vec\sigma)\bigr|&\leq {|R(0)|} + \frac{d}{N}\bigl(\sqrt{N}+\|\vec\sigma\|\bigr)\|\vec\sigma\|
        \leq d +
        d\bigl(1+\|\vec s\|\bigr)\|\vec s\|,\\
        \bigl|Q(\vec\sigma)\bigr|&\leq {|Q(0)|} + \frac{d}{N}\bigl(\sqrt{M}+2\sqrt{n}\|z\|+\|\bar G\|_{op}\|\vec\sigma\|\bigr) \|\bar G\|_{op} \|\vec\sigma\|\\ &\leq {d( \alpha +n \|\bar z\|^2)} + d\bigl(\sqrt{\alpha}+2\sqrt{n}\|\bar z\|+\|\bar G\|_{op}\|\vec s\|\bigr) \|\bar G\|_{op} \|\vec s\|,
    \end{align*}
    {where we used the assumption \eqref{phipsi}}. Using Lemma \ref{lem2} with $\lambda=\eta=0$, we see that there exists some $C$ such that
    $$
      \big\la \|\vec s\|^{2k}\big\ra_{0,0}\leq C \bigl(\|\bar z\|^{2k}+\|\bar G\|_{op}^{2k}+1+h^{2k}\bigr).
    $$
    By Corollary 5.35 in \cite{vershynin2010introduction},
    \begin{equation}
        \p(\|\bar G\|_{op}\geq 1+\sqrt{\alpha}+t)\leq 2\exp({-Nt^2/2}),
        \label{Gnormconc}
    \end{equation}
    which implies that $\e\|\bar G\|_{op}^m\leq C_m < \infty$ for any $m\geq 1$ and some constant $C_m$ independent of $N$. Similarly, by gaussian concentration for gaussian vectors, $\e\|\bar z\|^{m}\leq C_m <\infty$. Together with the above inequalities, this finishes the proof. \begin{flushright}$\square$\end{flushright}

% ----------------------------------------------------------------------------

\subsection{Concentration of the free energy}

For clarity, here we will now denote $F_N(\lambda,\eta)$ by $F(G,z)$. Let us start with the following result.
\begin{lemma}\label{concentrationfreeenergy1}
    There exist constants $\eps>0$ and $K>0$ such that for all $|\lambda|,|\eta|\leq \eps$, 
    \begin{align*}
      &\bigl|F(G,z)-F(G',z')\bigr|\\
      &\quad\leq K\Bigl(\|\bar z\|^3+\|\bar G\|_{op}^3+\|\bar z'\|^3+\|\bar G'\|_{op}^3+1+\|\bar G\|_{op}h^2+\|\bar G'\|_{op}h^2+h^{2}\Bigr)\bigl\|(\bar G-\bar G',\bar z-\bar z')\bigr\|.
    \end{align*}
\end{lemma}
\noindent
    \emph{Proof.} First of all,
  \begin{align*}
      \partial_{g_{ki}}F(G,z)&=\frac{1}{N^{3/2}}\sum_{\ell \le n}\Bigl\la\Bigl(u'\big(S_k(\sigma^\ell)+z_k\big)+{\eta}\partial_{x_\ell}\psi\bigl(S_k(\sigma^1)+z_k,\ldots,S_k(\sigma^n)+z_k\bigr)\Bigr)\sigma_i^\ell\Bigr\ra_{\lambda,\eta}.
  \end{align*}
  Using that $|u'(s)|\leq d(1+|s|)$ by (\ref{G1gen}) (below the constant $K$ amy change from place to place),
  \begin{align*}
      \sum_{k\le M}\sum_{i\le N}\Bigl\la\sum_{\ell \le n} u'\big(S_k(\sigma^\ell)+z_k\big)\sigma_i^\ell\Bigr\ra_{\lambda,\eta}^2
      &\leq  \Bigl\la
      \Bigl(\sum_{k\le M}\sum_{\ell \le n} u'\big(S_k(\sigma^\ell)+z_k\big)^2\Bigr)\sum_{i\le N}\sum_{\ell \le n}(\sigma_i^{\ell})^2\Bigr\ra_{\lambda,\eta}\\
      &\leq 3d\Bigl\la\Big(\sum_{k\le M}
      \sum_{\ell \le n} \bigl(1+S_k(\sigma^\ell)^2+z_k^2\bigr)\Big)\|\vec\sigma\|^2\Big\ra_{\lambda,\eta}\\
      &\leq 3d\Bigl\la 
      \bigl(nM+\|\bar G\|_{op}^2\|\vec\sigma\|^2+n\|z\|^2\bigr)\|\vec\sigma\|^2\Big\ra_{\lambda,\eta}\\
      &\leq KN^2 \bigl(\|\bar z\|^6+\|\bar G\|_{op}^6+1+\|\bar G\|_{op}^2h^{4}+h^4\bigr)
  \end{align*}
  where the last inequality used \eqref{lem2:eq1}. Note that from \eqref{bound}, there exists a constant $d'$ such that 
  \begin{align*}
      \bigl|\partial_{s_\ell}\psi(s)\bigr|&\leq d'\bigl(1+\|s\|\bigr).
  \end{align*}
  Using this inequality, we can argue similarly that
  \begin{align*}
      &\sum_{k\le M}\sum_{i\le N}\Bigl\la\sum_{\ell \le n}\partial_{x_\ell}\psi\bigl(S_k(\sigma^1)+z_k,\ldots,S_k(\sigma^n)+z_k\bigr)\sigma_i^\ell\Bigr\ra_{\lambda,\eta}^2\\
      &\qquad\leq {d'}^2\Bigl\la
      \Bigl(\sum_{k\le M}\sum_{\ell \le n} \Bigl[1+\Bigl(\sum_{\ell'=1}^n|S_k(\sigma^{\ell'})+z_k|^2\Bigr)^{1/2}\Bigr]^2\Bigr)\sum_{i\le N}\sum_{\ell \le n}(\sigma_i^{\ell})^{2}\Bigr\ra_{\lambda,\eta}\\
      &\qquad\leq 4{d'}^2n\Bigl\la
      \Bigl(\sum_{k\le M}\Bigl[1+\sum_{\ell'=1}^n(S_k(\sigma^{\ell'})^2+z_k^2)\Bigr]\Bigr)\|\vec\sigma\|^2\Bigr\ra_{\lambda,\eta}\\
      &\qquad\leq 4{d'}^2n\Bigl\la
      \Bigl(M+\|\bar G\|_{op}^2\|\vec\sigma\|^2+n\|z\|^2\Bigr)\|\vec\sigma\|^2\Bigr\ra_{\lambda,\eta}\\
      &\qquad\leq K'N^2\bigl(\|\bar z\|^6+\|\bar G\|_{op}^6+1+\|\bar G\|_{op}^2h^{4}+h^4\bigr).
  \end{align*}
  Putting these together yields that
  \begin{align*}
      \sum_{k\le M}\sum_{i\le N}|\partial_{g_{ki}}F(G,z)|^2&\leq \frac{K''}{N}\bigl(\|\bar z\|^6+\|\bar G\|_{op}^6+1+\|\bar G\|_{op}^2h^{4}+h^4\bigr).
  \end{align*}
  In similar manner, we readily compute that
  \begin{align*}
      \partial_{z_k}F(G,z)&=\frac{1}{N}\sum_{\ell \le n}\Bigl\la u'\big(S_k(\sigma^\ell)+z_k\big)+{\eta}\partial_{x_\ell}\psi\bigl(S_k(\sigma^1)+z_k,\ldots,S_k(\sigma^n)+z_k\bigr)\Bigr\ra_{\lambda,\eta}
  \end{align*}
  and 
  \begin{align*}
      \sum_{k\le M} |\partial_{z_k}F(G,z)|^2&\leq \frac{K'''}{N}\bigl(\|\bar z\|^4+1+h^{4}\bigr).
  \end{align*}
    All together, we get that
    $$
        \|\nabla F(G,z)\| \leq \frac{K_4}{\sqrt{N}}\bigl(\|\bar z\|^3+\|\bar G\|_{op}^3+1+\|\bar G\|_{op}h^{2}+h^2\bigr),
    $$  
    for some $\delta>0$, which implies Lemma \ref{concentrationfreeenergy1}. \begin{flushright}$\square$\end{flushright}

\begin{lemma}\label{lem:conc_fe}
  There exist absolute constants $\eps>0$ and $K>0$ such that, for all $|\lambda|,|\eta|\leq \eps$,
  \begin{align*}
  	%\label{concentrationfreeenergy:eq1}
      \mathrm{Var}F_N(\lambda,\eta)\leq \frac{K(1+h^{4})}{N}.
  \end{align*}
\end{lemma}
\noindent
    \emph{Proof.} For simplicity of notation, denote 
    $$
        y:=(G,z),\quad f(y):=F(G,z), \quad \mbox{ and }\quad r(y):=\|\bar G\|_{op}^3+\|\bar z\|^3+\|\bar G\|_{op}h^2.
    $$ 
    If we define $\bar y:=y/\sqrt{N}$, then the above Lemma proves that
    \begin{align*}
    	%\label{pseudolipA}
      \big|f(y)-f(y')\big|\leq K \bigl(r( y)+r( y')+1\bigr) \|\bar y-\bar y'\|,
  \end{align*}
    Let $L > 0$ be a constant to be chosen later and $L_h:=L(1+h^2)$. Since, for $r( y),r( y')\leq L_h$,
  \begin{align*}
      f(y)&\leq f(y')+K(2L_h+1)\|\bar y-\bar y'\|
  \end{align*}
    and we have equality for $y'=y$, if we define
  \begin{align*}
      \Phi(y)&:=\inf\bigl\{f(y')+K(2L_h+1)\|\bar y-\bar y'\| : r(y')\leq L_h\bigr\},
  \end{align*}
    then $\Phi(y)=f(y)$ for all $y$ such that $r(y)\leq L_h$. From this definition, it can be seen from the reverse triangle inequality that
  \begin{align}\label{eq3a}
      \big|\Phi(y)-\Phi(y')\big|&\leq K(2L_h+1) \|\bar y-\bar y'\| = \frac{K(2L_h+1)}{\sqrt N} \|y-y'\| 
  \end{align}
  and
  \begin{align}\label{eq4a}
      \big|\Phi(y)-f(y)\big|&\leq \bigl(|\Phi(y)|+|f(y)|\bigr)I(r(y)>L_h).
  \end{align}
    By (\ref{eq3a}) and the Gaussian-Poincar\'e concentration inequality,
    \begin{align}\label{eq5a}
      \e \bigl|\Phi(y)-\e\Phi(y)\bigr|^2\leq \e \|\nabla \Phi(y)\|^2\le   \frac{K^2 (2L_h+1)^2}{N}\le \frac{K'(1+ h^4)}{N}.
    \end{align}
    Next, using \eqref{eq4a} and the Cauchy-Schwarz inequality, we control 
    \begin{align*}
      \big|\e f(y)-\e \Phi(y)\big|&\le \e\big|f(y)-\Phi(y)\big|\\ & \leq\e\bigl(|\Phi(y)|+|f(y)|\bigr) I(r(y)>L_h)\\
      &\leq \sqrt{2\bigl(\e f(y)^2+\e \Phi(y)^2\bigr)\p(r(y)\geq L_h)}.
    \end{align*}
    Recalling that $r(y) = \|\bar G\|_{op}^3+\|\bar z\|^3+\|\bar G\|_{op}h^2$, by (\ref{Gnormconc}) and gaussian concentration for $\|\bar z\|$, for large enough $L>0$, we have $\p(r(y)\geq L_h)\leq 4 \exp(-N)$. On the other hand, {by Lemma \ref{concentrationfreeenergy1} with $(G,z) = y$ and $(G',z') = 0$ we have}
    \begin{align*}
        f(y)^2\,&{\leq 2f(0)^2 + 2 K^2 \Bigl(\|\bar z\|^3+\|\bar G\|_{op}^3+1+\|\bar G\|_{op}h^2+h^2\Bigr)^2\frac{\|y\|^2}{N}}\\ & {\leq 2f(0)^2 + 2 K^2 N^4 \Bigl(\frac{\|\bar z\|^3}{N^{3/2}}+\frac{\|\bar G\|^3}{N^{3/2}}+\frac{1+h^{2}}{N^{3/2}}+\frac{\|\bar G\|h^{2}}{N^{3/2}}\Bigr)^2 \Big(\frac{\| \bar G\|^2}{N} + \frac{\|\bar z\|^2}{N}\Big),}
    \end{align*}
    {and since $|f(0)|$ does not depend on $N$ and all the moments of $\|\bar z\|$ and $\|\bar G\|/\sqrt{N}$ are uniformly bounded, we have that $\e f(y)^2 \leq K'(1+ h^{4}) N^4$. Similarly,}
    \begin{align*}
        {\Phi(y)^2 \leq 2f(0)^2 + 2 K^2 N (2L_h+1)^2 \Big(\frac{\| \bar G\|^2}{N} + \frac{\|\bar z\|^2}{N}\Big),}
    \end{align*}    
    {from which we conclude that $\e \Phi(y)^2 \leq K''(1+h^{4}) N$.} Finally, using the bound
      \begin{align*}
          &\e \bigl|f(y)-\e f(y)\bigr|^2
          \leq 3\bigl(\e \bigl|f(y)-\Phi(y)\bigr|^2+\e \bigl|\Phi(y)-\e \Phi(y)\bigr|^2+\bigl|\e\Phi(y)-\e f(y)\bigr|^2\bigr)
      \end{align*}  
    finishes the proof. \begin{flushright}$\square$\end{flushright}

% ----------------------------------------------------------------------------

\subsection{Proof of Proposition~\ref{ThmOConc}}\label{SecOverConc}

We will only consider the case of the overlap $R(\vec\sigma)$, since the proof for $Q(\vec\sigma)$ is almost identical.
In this section we will denote by $K$ a constant that depends only on the parameters of the model, but can vary from one occurrence to the next. Let us consider the event
\begin{equation*}
    E_N:=\Bigl\{\|\bar G\|_{op}\leq K\Bigr\}\bigcap\, \Bigl\{\|\bar z\|\leq K\Bigr\}.
%\label{EventEN}
\end{equation*}
By (\ref{Gnormconc}), we can and will choose $K\geq 1$ so that the probability of the event $E_N$ is at least $1-4e^{-N}$. From Lemma \ref{lem3}, 
\begin{align}
    \e\Bigl[\Bigl\la \Bigl| R(\vec\sigma)-\e\la R(\vec\sigma)\ra_{0,0}\Bigr|^2 \Bigr\ra_{0,0}I({E_N^c})\Bigr]&\leq 4\Bigl(\e\la R(\vec\sigma)^4\ra_{0,0}\,\p(E_N^c)\Bigr)^{1/2}\leq K \exp(-N/2).
    \label{RcontEcom}
\end{align}
Let $F_N'$ be the $\lambda$-derivative of $F_N$. In the remainder we will control the same expectation on the event $E_N$, and we will use the convexity of $F_N(\lambda,\eta)$ in $\lambda$, which implies that 
\begin{align}
    &\Bigl|\la R(\vec\sigma)\ra_{0,0}-\e\la R(\vec\sigma)\ra_{0,0}\Bigr|
    =\Bigl|F_N'(0,0)-\e F_N'(0,0)\Bigr|\nonumber\\
    &\qquad\leq F_N'(\lambda,0)-F_N'(-\lambda,0)+\frac{1}{\lambda}\Bigl(|F_N(\lambda,0)-\e F_N(\lambda,0)|  \nonumber  
    \\
    &\qquad\qquad+|F_N(0,0)-\e F_N(0,0)|+|F_N(-\lambda,0)-\e F_N(-\lambda,0)|\Bigr).\label{RcontConv} 
\end{align}
Take $\eps>0$ such that, for $|\lambda|,|\eta|\leq \eps$, the inequalities in Lemmas~\ref{lem1.5}, \ref{lem2} and \ref{lem:conc_fe} hold. By Lemma~\ref{lem:conc_fe}, the expectation of the square of the term between the parentheses above can be bounded by $K/(N\lambda^2)$. 

In order to control the first term $F_N'(\lambda,0)-F_N'(-\lambda,0)$ above, our goal will be to control
\begin{align*}
    \partial_{\lambda\lambda}F_N(\lambda,\eta)&=N\Bigl\la\bigl|R(\vec\sigma)-\la R(\vec\sigma)\ra_{\lambda,\eta}\bigr|^2\Bigr\ra_{\lambda,\eta}.
\end{align*}
For simplicity of notation we denote $\vec s:= \vec\sigma/\sqrt{N}$. Since, for $|\lambda|,|\eta|\leq \eps$, the inequalities in Lemmas~\ref{lem1.5} and \ref{lem2} hold, on the event $E_N$, for some $L\geq 1$,
 \begin{align}
  \label{lem1.5:eq1ag}
  \big\la I( \|\vec s\|\geq L)\big\ra_{\lambda,\eta}\leq \exp(-KN)
  \quad\mbox{and}\quad
  \big\la \|\vec s\|^{10}\big\ra_{\lambda,\eta}\leq K.
\end{align}
Also, by Lemma \ref{lem:lip},
\begin{align}
    \big|R(\vec{\sigma})-R(\vec{\sigma}')\big|&\leq  D\Bigl(1+\|\vec s\|+\|\vec s'\|\Bigr)\|\vec s-\vec s'\|,
    \label{RlipDR}
\end{align}
where we also denoted $\vec s':= \vec\sigma'/\sqrt{N}$. 

We now proceed as in the proof of Lemma \ref{lem:conc_fe}. Take $L$ as in (\ref{lem1.5:eq1ag}) and $D$ as in (\ref{RlipDR}), denote $r(\vec\sigma):=\|\vec s\|$ and let
\begin{align*}
	%\label{Defrhospr}
  \rho(\vec\sigma)&:=\inf\bigl\{R(\vec\sigma')+D(2L+1)\| \vec s-\vec s'\|:r(\vec s')\leq L\bigr\}.
\end{align*}
Then, from (\ref{RlipDR}), $\rho(\vec\sigma) = R(\vec\sigma)$ for all $r(\vec\sigma)\leq L$, 
\begin{align}\label{eq3aag}
  |\rho(\vec\sigma)-\rho(\vec\sigma')|&\leq D\bigl(2L+1\bigr) \|\vec s-\vec s'\|,
    \\
    \label{eq4aag}
  |\rho(\vec\sigma)-R(\vec\sigma)|&\leq \big(|\rho(\vec\sigma)|+|R(\vec\sigma)|\big)I(r(\vec\sigma)>L).
\end{align}
Using \eqref{eq4aag} and the Cauchy-Schwarz and Minkowski inequalities,
\begin{align*}
  \Bigl\la|\rho(\vec\sigma)-R(\vec\sigma)|^2\Bigr\ra_{\lambda,\eta}&\leq
  \Bigl\la\bigl(|\rho(\vec\sigma)|+|R(\vec\sigma)|\bigr)^2I(r(\vec\sigma)>L)\Bigr\ra_{\lambda,\eta}
  \\
  &\leq 2\Bigl(
  \big\la \rho(\vec\sigma)^4\big\ra_{\lambda,\eta}^{1/2}
    + \big\la R(\vec\sigma)^4\big\ra_{\lambda,\eta}^{1/2} 
  \Bigr)
  \big\la I(r(\vec\sigma)>L)\big\ra_{\lambda,\eta}^{1/2}
  \\
    &\leq 2\Bigl(
  \big\la \rho(\vec\sigma)^4\big\ra_{\lambda,\eta}^{1/2}
    + \big\la R(\vec\sigma)^4\big\ra_{\lambda,\eta}^{1/2} 
  \Bigr) \exp(-KN),
\end{align*}
where in the last line we used the first equation in (\ref{lem1.5:eq1ag}). Since $R(0)=\phi(0)$, by (\ref{RlipDR})  
$$  
    |R(\vec{\sigma})|\leq |\phi(0)| + D\bigl(1+\|\vec s\|\bigr)\|\vec s\|.
$$
Since $\rho(0) = R(0)=\phi(0)$, by (\ref{eq3aag}),  
$$  
    |\rho(\vec{\sigma})|\leq |\phi(0)| + D\big(2L+1\big)\|\vec s\|.
$$
Therefore, the second equation in (\ref{lem1.5:eq1ag}) implies that
$$
    \bigl\la \rho(\vec\sigma)^4\bigr\ra_{\lambda,\eta}^{1/2}
    + \bigl\la R(\vec\sigma)^4\bigr\ra_{\lambda,\eta}^{1/2}
    \leq K
$$
and we get that
\begin{equation}
    \Bigl\la|\rho(\vec\sigma)-R(\vec\sigma)|^2\Bigr\ra_{\lambda,\eta}\leq K \exp(-KN).  
    \label{ATrhoR}
\end{equation}

By Lemma \ref{lem:convex}, on the event $E_N$, if $|\lambda|+ K |\eta|\leq \kappa/(d n)$ then function \eqref{lem:convex:eq1} 
% \begin{align}\label{lem:convex:eq1a}
%   \vec\sigma\mapsto\frac{\kappa}{2}\|\vec\sigma\|^2-\lambda NR(\vec\sigma)-\eta NQ(\vec\sigma)
% \end{align}
is convex, and thus the Hamiltonian $X_{N,\lambda,\eta}(\vec\sigma)$ in (\ref{HamXLE}) is strictly concave on $\Reals^{Nn}$ and its Hessian is bounded by $-\frac{\kappa}{2}\mathrm{Id}$. In this case, the Brascamp-Lieb inequality (see, e.g., \cite[Theorem 5.1]{brascampleb}) implies that
\begin{equation*}
    \Bigl\la\bigl|\rho(\vec\sigma)-\la \rho(\vec\sigma)\ra_{\lambda,\eta}\bigr|^2\Bigr\ra_{\lambda,\eta} \leq \frac{2}{\kappa} \Bigl\la \|\nabla \rho(\vec\sigma)\|^2\Bigr\ra_{\lambda,\eta}
    \leq  \frac{K}{N},
\end{equation*}
where the last inequality used (\ref{eq3aag}). Together with (\ref{ATrhoR}), this implies that, on the event $E_N$, 
\begin{align*}
    \partial_{\lambda\lambda}F_N(\lambda,\eta)&=N\Bigl\la\bigl|R(\vec\sigma)-\la R(\vec\sigma)\ra_{\lambda,\eta}\bigr|^2\Bigr\ra_{\lambda,\eta}\leq K,
\end{align*}
assuming as above that $|\lambda|,|\eta|\leq \eps$ and $|\lambda|+ K |\eta|\leq \kappa/(d n)$. In particular, 
\begin{equation}
    \Bigl\la\bigl|R(\vec\sigma)-\la R(\vec\sigma)\ra_{0,0}\bigr|^2\Bigr\ra_{0,0}\leq \frac{K}{N}
    \label{RcontEcomC}
\end{equation}
and $\bigl|F_N'(\lambda,0)-F_N'(-\lambda,0)\bigr|\leq K\lambda.$
Plugging this into (\ref{RcontConv}) and, as we mentioned above, using Lemma \ref{lem:conc_fe} for the second term, 
\begin{align*}
    \e\Bigl[\bigl|\la R(\vec\sigma)\ra_{0,0}-\e\la R(\vec\sigma)\ra_{0,0}\bigr|^2 I({E_N})\Bigr]
    &\leq K\Bigl(\lambda^2 +\frac{1}{N\lambda^2}\Bigr)\leq \frac{K}{\sqrt{N}},
\end{align*}
by optimizing over $\lambda$. Using (\ref{RcontEcomC}) on the event $E_N$ and then combining with (\ref{RcontEcom})  gives
\begin{align*}
    \e \Bigl\la \Bigl| R(\vec\sigma)-\e\la R(\vec\sigma)\ra_{0,0}\Bigr|^2 \Bigr\ra_{0,0}
    \leq \frac{K}{\sqrt{N}}.
\end{align*}
This proves the first inequality in (\ref{ThMRcond1}). The proof of the second inequality is identical. \begin{flushright}$\square$\end{flushright}

\section{Fixed point equations} \label{sec:5}
This section  establishes the proof of Theorem \ref{thm:conv_to_solution}. The key ingredient here is to show that equations \eqref{CPeqQ}--\eqref{CPeqRb} are satisfied by the subsequential limits of the overlaps, based on a cavity/leave-one-out argument and gaussian integration by parts.

%-------------------------------------------------------------------------------------------------------------------

\subsection{M-cavity}

For every measurable function $f:\Reals^{N n} \mapsto \Reals$, define the Gibbs expectation associated to the ``cavity in $M$'' by
\begin{equation*}
    \langle f \rangle_{c} := \frac{1}{Z_{c}^n} \int_{\Reals^{Nn}} \big(\prod_{\ell \le n} d\sigma^\ell\big)  f \exp{\sum_{\ell \leq n}H_{N,c}(\sigma^\ell)},
\end{equation*}
 where the M-cavity Hamiltonian is $$H_{N,c}(\sigma^\ell) := \sum_{k=2}^{M} u\bigl(S_k(\sigma^\ell)+z_k\bigr) + h \sum_{i\le N} \sigma_i^\ell - \kappa \|\sigma^\ell\|^2$$ and $Z_{c}$ is its corresponding partition function. In other words, the Hamiltonian here singles out the component $u (S_M(\sigma)+z_M )$ from $H_N(\sigma).$ It is easy to check that for every measurable $f : \Reals^{Nn} \mapsto \Reals$,
\begin{equation}\label{eq:cavity_mean_N}
    \langle f(\sigma^1,\ldots,\sigma^n)\rangle = \frac{\langle f(\sigma^1,\ldots,\sigma^n) \exp{\sum_{\ell\leq n}u(S_1(\sigma^\ell)+z_1)} \rangle_{c}}{\langle \exp{\sum_{\ell\leq n}u(S_1(\sigma^\ell)+z_1)} \rangle_{c}},
\end{equation}
where $\langle \,\cdot \, \rangle$ is the expectation with respect to the Gibbs measure proportional to $\exp \sum_{\ell\le n}H_N(\sigma^\ell)$ of the ``replicated system''.
The following two lemmas will play an essential role in establishing the fixed point equations in the next subsection.

\begin{lemma}\label{lem:uk_mean}
    For every $k\geq1$, there exists some $C_k > 0$ such that, uniformly on $N$, $$\e\big\langle |u(S_1(\sigma)+z_1)|^{k} \big\rangle \leq C_k.$$ 
\end{lemma}
\noindent
    \emph{Proof.} Here we will denote $u(S_1(\sigma)+z_1)$ by $u_1\le 0$. Because $|u_1|^k$ is an increasing function of $|u_1|$ and $\exp(-|u_1|)$ is decreasing, the Harris inequality implies $\langle |u_1|^k \exp(-|u_1|) \rangle_c \le \langle |u_1|^k\rangle_c \langle \exp(-|u_1|) \rangle_c$. As a consequence we have that
    \begin{equation*}
        \e\langle |u_1|^k \rangle = \e\frac{\langle |u_1|^k \exp(-|u_1|) \rangle_c}{\langle \exp(-|u_1|) \rangle_c} \leq \e\langle |u_1|^k \rangle_c.
    \end{equation*}
    Furthermore, because the standard gaussian vector $g_1 := (g_{11},\dots,g_{1N})$ is independent of $\sigma$ under the cavity measure, by the rotational invariance of $g_1$ we have that $\langle g_1, \sigma \rangle \sim g \|\sigma\|$ where $g$ is an independent standard gaussian variable, so under the cavity measure $S_1(\sigma)\sim g\|\sigma\|/\sqrt{N}=g\sqrt{R_{1,1}}$. Thus, by Assumption~\ref{ass:growth_u}
    \begin{equation*}
        \e\langle |u_1|^k \rangle_c \leq d^k \e\langle (1 + 2 g^2 R_{1,1} + 2 z_1^2)^k \rangle_c. 
    \end{equation*}
    Finally, by the independence of $R_{1,1}$ from $g$ and $z_1$ and Proposition~\ref{ThmOConc}, the right-hand side of this last equation is easily seen to be smaller than some constant that does not depend on $N$. \begin{flushright}$\square$\end{flushright}

\begin{lemma}\label{lem:eq_cavity_mean}
    There exist constants $K,K' > 0$ such that 
    \begin{equation*}
        \Big|\e\langle R_{1,1}\rangle_c - \e\langle R_{1,1}\rangle \Big| \leq \frac{K}{\sqrt{N}} \quad \mbox{and} \quad \Big|\e\langle R_{1,2}\rangle_c - \e\langle R_{1,2}\rangle \Big| \leq \frac{K'}{\sqrt{N}}.
    \end{equation*}
\end{lemma}
\noindent
    \emph{Proof.} As in the previous lemma, we will denote $u(S_1(\sigma)+z_1)$ by $u_1$. Let $$H_{N,t}(\sigma) := H_{N,c}(\sigma) + t u_1(\sigma)$$ for $t\in[0,1]$, and denote by $\langle\, \cdot \,\rangle_t$ the expectation with respect to the associated Gibbs measure. We then clearly have that
    \begin{equation*}
        \e\langle R_{1,1}\rangle_{c} = \e\langle R_{1,1}\rangle_{t=0}  \quad \mbox{and} \quad \e\langle R_{1,1}\rangle = \e\langle R_{1,1}\rangle_{t=1}.
    \end{equation*}
    Along this interpolation,
    \begin{equation*}
        \begin{split}
            \Big| \frac{d}{dt} \e\langle R_{1,1}\rangle_t \Big| & = \Big| \e\langle (R_{1,1}-  \langle R_{1,1}\rangle_t) u_1 \rangle_t \Big|  \leq \sqrt{\e\langle (R_{1,1}-  \langle R_{1,1}\rangle_t)^2\rangle_t} \times \sqrt{\e\langle u_1^2 \rangle_t} \leq \frac{K}{\sqrt{N}}\times K',
        \end{split}
    \end{equation*}
    for some constants $K,K' > 0$ that do not depend on $t\in[0,1]$. Here, the bound for the first factor follows from the Brascamp-Lieb inequality, while the bound on the second term is a generalization of Lemma~\ref{lem:uk_mean}. The proof for $R_{1,2}$ is similar. \begin{flushright}$\square$\end{flushright}

\subsection{Consistency equations for the overlaps} Throughout this subsection, we will repeatedly use the concentration results for the generalized overlaps $Q$ and $R$ from Section \ref{sec:4} that are applicable to the overlaps $Q_{1,1}$ using $\psi(s)=u''(s)+u'(s)^2$, $Q_{1,2}$ using $\psi(s_1,s_2)=u'(s_1)u'(s_2)$, $R_{1,1}$ using $\phi(s)=s^2$, and $R_{1,2}$ using $\phi(s_1,s_2)=s_1s_2$.
Denote $$\rho_N:=\e\langle R_{1,1}\rangle, \quad q_N:=\e\langle R_{1,2}\rangle, \quad \bar r_N := \e\langle Q_{1,1}\rangle \quad \mbox{and} \quad r_N := \e\langle Q_{1,2}\rangle.$$
 By Proposition~\ref{ThmOConc}, there exists some compact set $K$ such that for all $N\geq 1$, $(\rho_N,q_N,\bar r_N,r_N)$ belongs to $K$. Thus, there exists a subsequence $(N_m)_{m\ge 1}$ of system sizes along which $$(\rho_{m},q_{m},\bar r_{m}, r_{m}):=(\rho_{N_m},q_{N_m},\bar r_{N_m}, r_{N_m})\xrightarrow{m\to +\infty}(\rho,q,\bar r, r)$$ for some $(\rho,q,\bar r, r)$ in $K$. For the rest of this section, we work with this convergent subsequence and we aim to show that $(\rho,q,\bar r,r)$ satisfies \eqref{CPeqQ}--\eqref{CPeqRb} in Propositions \ref{prop:hat_fix_point} and \ref{prop:til_fix_point}. To begin with, we establish

\begin{lemma}\label{lem:gaussian_approx}
Let, for $n\geq1$, $f : \Reals^{n} \mapsto \Reals$ be continuous and bounded and 
\begin{align}
\theta_\ell := \sqrt{\rho-q}\,\xi_\ell+\sqrt{q+\Delta_*}\,\tilde z \quad \mbox{for} \quad \ell\le n, \label{thetal}    
\end{align}
with $\tilde z,(\xi_\ell)_{\ell \le n}$ i.i.d. standard gaussian variables. Under Assumption~\ref{ass:growth_u}, we have that
    \begin{equation*}
        \e\big\langle f(S_1(\sigma^1)+z_1,\dots,S_1(\sigma^n)+z_1)\big\rangle_c\xrightarrow{m\to+\infty} \e f(\theta_1,\dots,\theta_n).
    \end{equation*}
\end{lemma}
\noindent
    \emph{Proof.} Along the proof $\ell$ and $\ell'$ always belong to $\{1,\ldots,n\}$. Recall $S_1^\ell:=S_1(\sigma^\ell)$. Because $f$ is continuous and bounded and because for fixed values of $\sigma^1,\dots,\sigma^\ell$ the quantity $(S_1^1+z_1,\dots,S_1^n+z_1)$ is a gaussian vector with covariance matrix $\Sigma_m$ with entries $\Sigma_{m,\ell \ell'} = R_{\ell, \ell'} + \Delta_*$, we have that there exists some continuous and bounded $g: \Reals^{n\times n} \mapsto \Reals$ such that
    \begin{equation}\label{eq:gaussian_mean}
        \e\big\langle f(S_1^1+z_1,\dots,S_1^n+z_1)\big\rangle_c = \e\langle g(\Sigma_m) \rangle_c.
    \end{equation}
    By Proposition~\ref{ThmOConc} and Lemma~\ref{lem:eq_cavity_mean}, $\Sigma_m \xrightarrow{L^2} \Sigma$ as $m\to+\infty$ where $$\Sigma_{\ell \ell'} = (\rho+\Delta_*) {I}(\ell = \ell') + (q+\Delta_*) I(\ell\neq \ell').$$ Then, by the continuous mapping theorem, $\e\langle g(\Sigma_m) \rangle_c \to g(\Sigma)$. Observe that $\Sigma = D D^\intercal$, where
    \begin{equation*}
        D_{k k'} = \sqrt{q+\Delta_*} \ {I}(k'=1) + \sqrt{\rho-q} \ {I}(k'=k+1) \in \mathbb{R}^{n\times n+1}.
    \end{equation*}
    We can then define i.i.d. standard gaussian variables $v := (\tilde z,\xi_1,\dots,\xi_n)$ so that $\theta_\ell := (D v)_{\ell}$ is \eqref{thetal}, and then by reversing relation \eqref{eq:gaussian_mean} we get $g(\Sigma) = \e f(\theta_1,\dots,\theta_n)$. \begin{flushright}$\square$\end{flushright}

\begin{proposition}\label{prop:hat_fix_point}
 The vector $(\rho,q,\bar r,r)$ satisfies   \eqref{CPeqR} and \eqref{CPeqRb}.
\end{proposition}
\noindent
    \emph{Proof. } Let $(\theta_\ell)_{\ell\le n}$ as in \eqref{thetal}, $\theta=\theta_1$, $\xi=\xi_1$ and $a_\ell:=S_1(\sigma^\ell)+z_1$. Let $K\geq1$. For each $k\le K$ define {(we let $\prod_{j=2}^1 (\cdots) =1$)}
    \begin{equation*}
        f_{k}(a_1,\dots,a_k) := ({u'}(a_1)^2 + {u''}(a_1)) \exp{u(a_1)} \prod_{j=2}^{k} (1 - \exp{u(a_j)}).
    \end{equation*}
      By equation \eqref{eq:cavity_mean_N}, the overlap definition \eqref{overlapsQ}, the spin symmetry and the geometric series (recall $u(s)\le 0$) we have that
    \begin{align}
        \e\langle{Q_{1,1}}\rangle &= \alpha \e\langle{u'}(a_1)^2 + {u''}(a_1)\rangle  \nonumber\\
        &=\alpha \e\frac{\langle({u'}(a_1)^2 + {u''}(a_1))\exp u(a_1)\rangle_c}{\langle\exp u(a_1)\rangle_c}  \nonumber\\
        &= \alpha \sum_{k\le K} \e \langle f_k(a_1,\dots,a_k) \rangle_c + \alpha \e  \big(\langle {u'}(a_1)^2 + {u''}(a_1) \rangle(1 -\langle \exp{u(a_1)} \rangle_c )^{{K}}\big).\label{eq:geom_series}
    \end{align}
    Note that by Assumption~\ref{ass:growth_u} the first term is the mean of a continuous and bounded function of $a_1,\dots,a_K$. Then by Lemma~\ref{lem:gaussian_approx} we have that
    \begin{equation*}
        \sum_{k\leq K} \e \langle f_k(a_1,\dots,a_k) \rangle_c \xrightarrow{m\to+\infty} \sum_{k\leq K} \e f_k(\theta_1,\dots,\theta_k).
    \end{equation*}
    The second term can be bounded in the following way
    \begin{align*}
            &\Big|\e \left( \langle {u'}(a_1)^2 + {u''}(a_1) \rangle(1 -\langle \exp{u(a_1)} \rangle_c )^K\right)\Big| \\
            &\qquad \leq \e \left(\langle d +d^2 (1+ {\sqrt{|u(a_1)|}})^2 \rangle(1 -\langle \exp{u(a_1)} \rangle_c )^{{K}}\right)\\
            & \qquad\leq \sqrt{\e \langle [d +d^2 (1+ {\sqrt{|u(a_1)|}})^2]^2\rangle \e \langle (1 - \exp{u(a_1)} )^{{2K}}\rangle_c }\\
            & \qquad\leq \sqrt{C\, \e \langle (1 - \exp{u(a_1)} )^{{2K}}\rangle_c},
    \end{align*}
    where in the first line we used Assumption \ref{ass:growth_u}, in the second Cauchy-Schwarz and Jensen's inequalities, and in the last one that by Lemma \ref{lem:uk_mean} there exists a constant $C > 0$ that bounds the first factor in the square root. Now, because $(1 - \exp{u(a_1)} )^{{2K}}$ is a continuous and bounded function of $a_1$, by Lemma~\ref{lem:gaussian_approx} we have that
    \begin{equation*}
        \e \langle (1 - \exp{u(a_1)} )^{2K}\rangle_c \xrightarrow{m\to+\infty} \e  (1 -\exp{u(\theta)} )^{2K}.
    \end{equation*}
    
    By taking the limit $m\to+\infty$ in equation \eqref{eq:geom_series} and using the convergence of the overlap $Q_{1,1}$ along the subsequence $(N_m)_{m\ge 1}$, this implies that
    \begin{equation*}
        \bar r = \alpha \sum_{k\leq K} \e f_k(\theta_1,\dots,\theta_k) + \mathcal{O}\Big(\sqrt{\e  (1 -\exp{u(\theta)} )^{2K}}\Big).
    \end{equation*}
    Consider the two terms in $f_k$ separately. Taking the limit $K\to+\infty$, by the monotone convergence theorem,
    \begin{align*}
         \sum_{k\leq K} \e {u'}(\theta)^2 \exp{u(\theta)} \prod_{j=2}^k(1 - \exp{u(\theta_j)})  &=  \e_{\tilde z}\,\Big[ \e_{\xi}{u'}(\theta)^2 \exp{u(\theta)} \sum_{k\leq K} (1 - \e_{\xi} \exp{u(\theta)})^{k-1}\Big] \\
         &\qquad\xrightarrow{K\to+\infty}  \e_{\tilde z}\, \frac{\e_\xi {u'}(\theta)^2 \exp u(\theta) }{\e_\xi  \exp u(\theta) },
    \end{align*}
    and since $|u''(a_1)| \leq d$ by Assumption~\ref{ass:growth_u}, by the dominated convergence theorem, we also have
    \begin{equation*}
        \sum_{k\leq K} \e {u''}(\theta) \exp{u(\theta)} \prod_{j=2}^k (1 - \exp{u(\theta_j)}) \xrightarrow{K\to+\infty} \e_{\tilde z}\, \frac{\e_\xi {u''}(\theta) \exp u(\theta)}{\e_\xi  \exp u(\theta)}.
    \end{equation*}
    To see that the second term goes to $0$, note that $1 - \exp{u(\theta)} < 1$. The limit then follows by the dominated convergence theorem. This proves $\bar r = \bar \Phi(q,\rho)$, i.e., equation \eqref{CPeqRb}.
    
    {The proof that $r = \Phi(q,\rho)$ (equation \eqref{CPeqR}) follows in a similar way. Let $K\geq1$. For each $k\le K$ define
    \begin{equation*}
       g_{k}(a_1,\dots,a_{2k}) := {u'}(a_1) {u'}(a_2) \exp(u(a_1)u(a_2)) \prod_{j=2}^{k} \big(1 - \exp(u(a_{2j-1})+u(a_{2j}))\big).
    \end{equation*}
    By Assumption~\ref{ass:growth_u}, all these are continuous and bounded functions of $a_1,\dots,a_{2k}$. As before, by the cavity formula and the geometric series we have that
    \begin{equation*}
        \e\langle{Q_{1,2}}\rangle = \alpha \sum_{k\le K} \e \langle g_k(a_1,\dots,a_{2k}) \rangle_c + \alpha \e  \big(\langle {u'}(a_1){u'}(a_2) \rangle(1 -\langle \exp{(u(a_1)+u(a_2))} \rangle_c )^{{K}}\big).
    \end{equation*}
    Finally, the conclusion follows in a similar manner as for the first equation: Lemma \ref{lem:gaussian_approx} is used for the sum while the second term can be seen to go to $0$ when $K\to+\infty$.}
    \begin{flushright}$\square$\end{flushright}

% Before establishing the second set of fixed point equations, we will define the following functions
% \begin{equation*}
%     \bar \Psi(r,\bar r) := \frac{1}{2 \kappa+ r-\bar r} + \frac{r + h^2}{(2\kappa + r - \bar r)^2} \quad \mbox{and} \quad \Psi(r,\bar r) := \frac{r + h^2}{(2\kappa + r - \bar r)^2}.
% \end{equation*}

To establish the second pair of fixed point equations, a feasible approach would be to perform a cavity argument over the variables indexed by $i=1,\ldots,N$ analogous to the proof of Proposition \ref{prop:hat_fix_point} (this ``cavity in $N$'' is in general more involved than the ``cavity in $M$''). Nevertheless, by the virtue of concentration of the overlaps, we show that they can also be obtained through an elementary derivation via gaussian integration by parts. 

\begin{proposition}\label{prop:til_fix_point}
 The vector $(\rho,q,\bar r,r)$ satisfies \eqref{CPeqQ} and \eqref{CPeqRho}.
\end{proposition}
\noindent
    \emph{Proof.} To see this, note that by the term $-\kappa (\sigma_1^1)^2$ in the Hamiltonian \eqref{HamHN} we can use gaussian integration by parts with respect to $\sigma_1^1$ and the gaussian variables $(g_{k1})_{k\leq M}$. In particular
    \begin{align*}
     \e\langle f((\sigma^\ell),g_{ki}) g_{ki}\rangle=\e\langle  \partial_{g_{ki}}f\rangle+\sum_{\ell}\big(\e\langle  f\partial_{g_{ki}}H_N(\sigma^\ell) \rangle-\e\langle  f\rangle \langle \partial_{g_{ki}} H_N(\sigma^\ell) \rangle\big),
    \end{align*}
    and, denoting the modified Hamiltonian $\hat H_N(\sigma):= \sum_{k\leq M} u(S_k(\sigma)+z_k) - h \sum_{i\leq N} \sigma_i$, 
    \begin{align*}
     2\kappa\e\langle f(\sigma_{i}) \sigma_{i}\rangle=\e\langle  \partial_{\sigma_{i}}f\rangle+\e\langle  f \partial_{\sigma_{i}}\hat H_N(\sigma)\rangle.
    \end{align*}
    These identities imply (recall the definitions \eqref{defs:AB})    
    \begin{align*}
        2 \kappa \e\langle{\sigma_1}\rangle &= \e\Big\langle{\frac1{\sqrt{N}}\sum_{k\leq M} A_k g_{k1}}-h\Big\rangle \\
        &= \e\Big\langle\frac{\sigma_1^1}{N}\sum_{k\leq M} \big((A_k^1)^2 + B_k^1 - A_k^1 A_k^2 \big)-h\Big\rangle,\\
        2 \kappa \e\langle(\sigma_1)^2\rangle & = \e\Big\langle 1{+\sigma_1\Big(\frac1{\sqrt{N}} \sum_{k\leq M} A_k g_{k1}-h \Big)}\Big\rangle\\
        & = \e\Big\langle1+\frac{(\sigma_1^1)^2}{N}\sum_{k\leq M} \big((A_k^1)^2 + B_k^1\big) -\frac{\sigma_1^1\sigma_1^2}{N}\sum_{k\leq M} A_k^1 A_k^2 - h\sigma_1\Big\rangle,\nonumber\\
        2 \kappa \e\langle\sigma^1_1\sigma^2_1\rangle & = \e\Big\langle\sigma_1^2{\Big(\frac1{\sqrt{N}}\sum_{k\leq M} A^1_k g_{k1}-h\Big)}\Big\rangle\\
        & = \e\Big\langle \frac{\sigma_1^1\sigma_1^2}{N}\sum_{k\leq M} \big((A_k^1)^2{+B_k^1}\big)+\frac{(\sigma_1^2)^2}{N}\sum_{k\leq M} A_k^1 A_k^2- 2\frac{\sigma_1^2\sigma_1^3}{N}\sum_{k\leq M} A_k^1 A_k^3  - h\sigma_1^2\Big\rangle.
    \end{align*}
    The concentrations of $R_{1,1}$, $R_{1,2}$, $Q_{1,1}$, and $Q_{1,2}$ (see Proposition~\ref{ThmOConc}) then implies that
    {\begin{align*}
        \e\langle \sigma_1\rangle &= - \frac{h}{2 \kappa +r_{m}-\bar r_{m}} + o_m(1),\\
        2 \kappa \rho_{m} &= 1 + \rho_{m} \bar r_{m} - q_{m} r_{m} + \frac{h^2}{2\kappa + r_{m} - \bar r_{m}} + o_m(1),\\
        2 \kappa q_{m} &= - q_{m} (r_{m} - \bar r_{m}) + (\rho_{m}-q_{m}) r_{m}+ \frac{h^2}{2\kappa + r_{m} - \bar r_{m}} + o_m(1).
    \end{align*}
    Sending $m\to+\infty$, these lead to $q = \Psi(r,\bar r)$ and $\rho= \bar \Psi(r,\bar r)$ (equations \eqref{CPeqQ} and \eqref{CPeqRho}) and complete our proof.
    \begin{flushright}$\square$\end{flushright}

Propositions \ref{prop:hat_fix_point} and \ref{prop:til_fix_point} together conclude that $(\rho,q,\bar r,r)$ satisfies \eqref{CPeqQ}--\eqref{CPeqRb}.  We are now in position to prove Theorem \ref{thm:conv_to_solution} for the ST spin glass model.

\subsection{Proof of Theorem~\ref{thm:conv_to_solution}.}  Consider any subsequence of system sizes. As explained above Lemma~\ref{lem:gaussian_approx}, there exists some further subsequence $(N_{m})_{n\geq1}$ such that
    \begin{equation*}
        (\rho_{N_{m}},q_{N_{m}},\bar r_{N_{m}},r_{N_{m}}) \xrightarrow{m\to+\infty}(\rho,q,\bar r,r)
    \end{equation*}
    for some $(\rho,q,\bar r,r)$ in a compact set $K$. By Propositions \ref{prop:hat_fix_point} and \ref{prop:til_fix_point} and the continuity of the functions $\bar \Phi$, $\Phi$, $\bar \Psi$, and $\Psi$, we see that $(\rho,q,\bar r,r)$ is a solution of \eqref{CPeqQ}--\eqref{CPeqRb}. Since this system of equations has only one unique solution as ensured by the given assumption, it follows that any convergent subsequence will share the same limit and this implies the assertion.

\section{Proof of Theorem \ref{ThmMain}} \label{sec:freeenergy}
In this section, we establish the convergence of $F_N$ using a basic interpolation in this section. Let $\alpha,\Delta_*,\Delta,\kappa > 0$ be fixed.  For $t\in[0,1]$, consider an auxiliary free energy
\begin{equation*}
    \phi_N(t) := \frac{1}{N}\e\ln \int  \exp\Big(\sum_{k\leq M} u\big(\sqrt{t}((\bar G \sigma)_k +  z_k)\big) - h \sum_{i\leq N} \sigma_i - \kappa \|\sigma\|^2\Big)\, d\sigma. 
\end{equation*}
Let $\e\la \,\cdot\,\ra_t$ be the joint mean with respect to the gaussian $(G,z)$ and Gibbs measure associated with the Hamiltonian \eqref{HamHN} after replacing $u(x)$ by $u(\sqrt{t}x)$. We need the following crucial lemma.

\begin{lemma}\label{lem:add1}
For any $t_0\in (0,1],$ there exist constants $K>0$ and $\delta>0$ such that for all $t\in [t_0-\delta,t_0+\delta]\cap(0,1]$ and $N\geq 1$ we have
\begin{align*}
\max\Bigl(	\Bigl|\frac{d}{dt}\e \la R_{1,1}\ra_t\Big|,	\Bigl|\frac{d}{dt}\e \la R_{1,2}\ra_t\Big|,	\Bigl|\frac{d}{dt}\e \la Q_{1,1}\ra_t\Big|,	\Bigl|\frac{d}{dt}\e \la Q_{1,2}\ra_t\Big|\Bigr)\leq K.
\end{align*} 

\end{lemma}

\begin{proof}
	We establish this inequality only for $\frac{d}{dt}\e \la R_{1,1}\ra_t$ as the others can be treated similarly. Let
	$$
	L(\sigma):=\frac{1}{N}\sum_{k\le M}\big(S_k(\sigma)+z_k\big)u'\big(\sqrt{t}(S_k(\sigma)+z_k)\big).
	$$
	Write
	\begin{align*}
		\frac{d}{dt}\mathbb{E}\langle R_{1,1}\rangle_t
		&=\frac{N}{2\sqrt{t}}\,\mathbb{E}\big\langle R_{1,1}(L(\sigma^1)-L(\sigma^2))\big\rangle_t\\
        &=\frac{N}{2\sqrt{t}}\,\mathbb{E}\big\langle (R_{1,1}-\langle R_{1,1}\rangle_t)(L(\sigma^1)-\langle L(\sigma^1)\rangle_t)\big\rangle_t.
	\end{align*}
	As in Section \ref{SecOverConc},
	$$
	\mathbb{E}\big\langle(R_{1,1}-\langle R_{1,1}\rangle_t)^2\big\rangle_t\leq \frac{K}{N} 
	\quad \mbox{and}\quad 
	\mathbb{E}\big\langle (L(\sigma^1)-\langle L(\sigma^1)\rangle_t)^2\big\rangle_t\leq \frac{K}{N},
	$$
	so, by the Cauchy-Schwarz inequality, 
	$
	\frac{d}{dt}\mathbb{E}\langle R_{1,1}\rangle_t\leq K.
	$
	Since this bound is valid locally uniformly in $t$, the assertion follows.
\end{proof}

    We now turn to the proof of Theorem \ref{ThmMain}. Let $(q,\rho,\bar r,r)$ be the unique solution to \eqref{CPeqQ}--\eqref{CPeqRb}. As explained in Section \ref{sec:3}, $F(q,\rho,r,\bar r)=F(q,\rho).$ Also, from Lemma \ref{lem:conc_fe}, $\lim_{N\to\infty}\mbox{Var}(F_N)=0.$ Our proof will be complete if we can show that all subsequential limits of $\e F_N$ are equal to $F(q,\rho,r,\bar r)$. We establish this statement as follows.
    
   First of all, from Lemma \ref{lem:add1}, we can use the Arzela-Ascoli theorem and a diagonalization process to further pass to a subsequence along which $\e\la R_{1,1}\ra_t$, $\e\la R_{1,2}\ra_t$, $\e\la Q_{1,1}\ra_t$ and $\e\la Q_{1,2}\ra_t$ converge to some $\rho_t,q_t,\bar r_t,r_t$ uniformly with respect to $t\in[0,1]$, respectively. For the rest of the proof, to lighten our notation, we shall simply assume that the convergences of these overlaps and $\e F_N$ are valid along the original sequence $N.$
   Here, $\rho_t,q_t,\bar r_t,r_t$ are defined on $t\in[0,1]$ and are locally Lipschitz on $(0,1].$ Furthermore, from Theorem~\ref{thm:conv_to_solution}, they satisfy \eqref{CPeqQ}--\eqref{CPeqRb} for $u(x)$ being replaced by $u(\sqrt{t}x)$ for all $t\in [0,1]$. 
    Now, from the uniqueness of the solution to the fixed point equations \eqref{CPeqQ}--\eqref{CPeqRb} (corresponding to $t=1$), 
   \begin{align}\label{add:eq1}
   	(\rho_1,q_1,\bar r_1,r_1)=(\rho,q,\bar r,r).
   \end{align}  When $t=0,$ 
    \begin{align}\label{add:eq2}
    \rho_0=\frac{1}{2\kappa}+\frac{h^2}{(2\kappa)^2},\quad q_0=\frac{h^2}{2\kappa},\quad  \bar{r}_0=0,\quad r_0=0.
    \end{align}
    Sending $t\downarrow 0$ in \eqref{CPeqQ}--\eqref{CPeqRb} also yields that $\lim_{t\downarrow 0}(\rho_t,q_t,\bar r_t,r_t)=(\rho_0,q_0,\bar r_0,r_0)$. Hence, $\rho_t,q_t,\bar r_t,r_t$ are continuous on $[0,1]$.

   Next, set
    \begin{equation*}
    	{f(t) := \alpha\e \ln\e_\xi \exp{u(\sqrt{t}\theta_t)}  +\frac12\Big(\frac{r_t+h^2}{2\kappa + r_t-\bar r_t} - \ln(2\kappa +r_t-\bar r_t) +  r_t q_t-\bar r_t \rho_t+\ln2\pi\Big),}
    \end{equation*}
    where $\theta_t := \tilde z \sqrt{\Delta_*+q_t}+\xi\sqrt{\rho_t-q_t}$ with i.i.d. standard gaussian $\tilde z, \xi$. Because the fixed point equations satisfy the critical point conditions
    \begin{equation*}
    	\frac{\partial f}{\partial \rho_t} =  \frac{\partial f}{\partial q_t} = \frac{\partial f}{\partial\bar r_t} = \frac{\partial f}{\partial r_t} = 0
    \end{equation*}
and $\rho_t,q_t,\bar r_t,r_t$ are locally Lipschitz on $(0,1),$ we conclude  that $df/dt = \partial f/ \partial t$. To compute this partial derivative, let $\theta_t' := \tilde z \sqrt{\Delta_*+q_t}+\xi'\sqrt{\rho_t-q_t}$ for $\xi'$ an independent copy of $\xi$. Using gaussian integration by parts with respect to $\tilde z$ and $\xi$ yields
\begin{equation*}
	f'(t) = \frac{\partial f}{\partial t} = \frac{\alpha}{2\sqrt{t}}\, \e_{\tilde z}\, \frac{\e_\xi u'(\sqrt{t}\theta_t)\exp u(\sqrt{t}\theta_t)\theta_t}{\e_{\xi'} u'(\sqrt{t}\theta_t')\exp u(\sqrt{t}\theta_t')}= {\frac{1}{2}} (\bar r_t (\rho_t+\Delta_*) - r_t (q_t+\Delta_*)) .
\end{equation*}
 On the other hand, note that
 \begin{equation}\label{eq:conv_deriv}
 	\phi_N'(t) = {\frac{1}{2}} \e\big\langle Q_{1,1} (R_{1,1}+\Delta_*) - Q_{1,2} (R_{1,2}+\Delta_*) \big\rangle_t,
 \end{equation}
 where we used gaussian integration by parts with respect to the elements of $\bar G$ and $z$ similar to the proof of Proposition~\ref{prop:til_fix_point}. 
 By Theorem \ref{thm:conv_to_solution}, for every $t\in[0,1]$, $(R_{1,1},R_{1,2},Q_{1,1},Q_{1,2})$ converges to $(\rho_t,q_t,\bar r_t,r_t)$ under $\e\la \,\cdot\,\ra_t$
and, hence,
 $$
\lim_{N\to+\infty} \phi_N'(t)= \frac{1}{2} (\bar r_t (\rho_t+\Delta_*) - r_t (q_t+\Delta_*)).
 $$

 To complete the proof, note that $\phi_N(1)=\e F_N$ and that from \eqref{add:eq1} and \eqref{add:eq2},
    \begin{align*}
        f(0)& = \alpha u(0)+\frac{h^2}{4\kappa} + \frac{1}{2}\ln\frac{\pi}{\kappa}=\phi_N(0),\\
        f(1)&=F(q,\rho,\bar r,r).
    \end{align*}
     Consequently, from \eqref{eq:conv_deriv} and noting that $\e\la R_{1,1}\ra_t,$ $\e\la R_{1,2}\ra_t,$ $\e \la Q_{1,1}\ra_t,$ and $\e\la Q_{1,2}\ra_t$ are uniformly bounded on $[0,1]$ due to Lemma \ref{lem3}, the dominated convergence theorem implies
     \begin{align*}
	|\e F_N-F(q,\rho,r,\bar r)|=|\phi_N(1)-f(1)|\leq \int_0^1 |\phi'_N(t)-f'(t)|dt\to 0,
    \end{align*}
completing our proof. \begin{flushright}$\square$\end{flushright}

\section{Proof of Theorem \ref{add:thm1}} \label{sec:7}
We now provide the proof of the main theorem for the regression task, from the results obtained for the ST model. Throughout the entire section, the assumptions in Theorem~\ref{add:thm1} are in force. Recall $h_N:=2\kappa \sqrt{\gamma_N}$ and $h:=2\kappa \sqrt{\gamma}.$ From the convergence \eqref{convNorm} we have $\e|h_N-h|^2\to 0$. Throughout this section, we denote by $F_N(h),F_N(h_N)$ and $\la\, \cdot\,\ra_{h},\la\, \cdot\,\ra_{h_N}$ the free energies $F_N$ and Gibbs expectations $\la\, \cdot\,\ra$, defined by \eqref{add:eq-2}, \eqref{gibbsExpec}, associated to the deterministic $h$ and randomized $h_N,$ respectively. In particular $\la\, \cdot\,\ra_{h}$ is the same as the plain $\la\, \cdot\,\ra$, while when defining $F_N(h_N)$ and $\la\, \cdot\,\ra_{h_N}$ the constant $h$ is replaced in the Hamiltonian \eqref{HamHN} by the random $h_N$. In the latter case, we also assume that $h_N$ is independent of all other randomness. First of all, we establish the following lemma:

\begin{lemma}[Controlling the randomness $h_N$] The free energy and overlap for the ST model with a random $h_N$ and fixed $h$ are asymptotically the same:
\begin{align}\label{add:eq-6}
	\lim_{N\to\infty}\e\bigl|F_N(h_N)-F_N(h)\bigr|=0
\end{align}
and
\begin{align}\label{add:eq-7}
	\lim_{N\to\infty}\e\la R_{1,2} \ra_{h_N}=\lim_{N\to\infty}\e\la R_{1,2}\ra_h=q.
\end{align}
\end{lemma}

\begin{proof}
For $0\leq t\leq 1,$ consider the interpolated free energy,
$$
F_{N,t}:=\frac{1}{N}\ln\int \exp\Bigl(\sum_{k\le M}u\big(S_k(\sigma)+z_k\big)-\big(th_N+(1-t)h\big)\sum_{i\le N}\sigma_i-\kappa \|\sigma\|^2\Bigr)\,d\sigma.
$$
Let $H_{N,t}(\sigma)$ be the Hamiltonian defined by the above exponent. The associated Gibbs mean is denoted by $\la \,\cdot\,\ra_t$ and is defined by
\begin{align}
\big\la f((\sigma^\ell)_{\ell \le L})\big\ra_t:=\frac{1}{(Z_{N,t})^L}\int f((\sigma^\ell)_{\ell \le L})\prod_{\ell \le L}\exp H_{N,t}(\sigma^\ell)\,d\sigma^\ell
\end{align}
with $Z_{N,t}:=\int \exp H_{N,t}(\sigma)\,d\sigma$ the corresponding normalization. Clearly $F_{N,1}=F_N(h_N)$ and $F_{N,0}=F_N(h)$, as well as $\la R_{1,2}\ra_{t=1}=\la R_{1,2}\ra_{h_N}$ and $\la R_{1,2}\ra_{t=0}=\la R_{1,2}\ra_{h}$. Let $M_\ell:=N^{-1}\sum_{i\le N}\sigma_i^\ell.$ A direct differentiation gives that
\begin{align*}
	\frac{d}{dt}F_{N,t}=(h-h_N)\la M_1\ra_t
\end{align*}
and
\begin{align*}
	\frac{d}{dt}\la R_{1,2}\ra_{t}&=2N(h-h_N)\la R_{1,2}(M_1-M_3)\ra_t\\
	&=2N(h-h_N)\la (R_{1,2}-\la R_{1,2}\ra_t)(M_1-\la M_1\ra_t)\ra_t.
\end{align*}
From these, the mean value theorem and the H\"older inequality, there exists $t\in[0,1]$ such that
\begin{align}\label{add:eq-15}
\e \big|F_N(h_N)-F_N(h)\big|&\leq \bigl(\e |h_N-h|^2\bigr)^{1/2}\bigl(\e\la M_1^2\ra_t\bigr)^{1/2}\nonumber\\
&\leq  \bigl(\e |h_N-h|^2\bigr)^{1/2}\Bigl(\frac{\e \la \|\sigma^1\|^2\ra_t}{N}\Bigr)^{1/2}
\end{align}
and some $t\in[0,1]$ such that
\begin{align}\label{add:eq-8}
&\bigl|\e\la R_{1,2}\ra_{h_N}-\e \la R_{1,2}\ra_h\bigr|\nonumber\\
&\qquad\leq 2N\bigl(\e|h_N-h|^2\bigr)^{1/2}\Bigl(\e\bigl\la \bigl|R_{1,2}-\la R_{1,2}\ra_t\bigr|^4\bigr\ra_t\Bigr)^{1/4}
\Bigl(\e\bigl\la \bigl|M_{1}-\la M_{1}\ra_t\bigr|^4\bigr\ra_t\Bigr)^{1/4}.
\end{align}
The same proof as of Lemma \ref{lem2} yields that for any $k\geq 1$ there exists some $K>0$ such that for all $t\in [0,1]$ and $N\geq 1,$
\begin{align}\label{add:eq-10}
	\la \|\sigma\|^{2k}\ra_t&\leq KN^k\bigl(\|\bar z\|^{2k}+\|\bar G\|_{op}^{2k}+h^{2k}+h_N^{2k}+1\bigr).
\end{align}
Since the expectation of every term in the bracket is bounded by a universal constant, we obtain \eqref{add:eq-6} by using \eqref{add:eq-15} and taking $k=1.$

As for \eqref{add:eq-7}, note that for any fixed $w\in \mathbb{R}^N,$  $f(\sigma):=N^{-1}w\cdot \sigma$ is Lipschitz with Lipschitz constant $\|w\|/N$ and that the Hessian of the interpolated Hamiltonian $H_{N,t}(\sigma)$ in $F_{N,t}$ is bounded from above by $-\kappa {\rm Id}/2.$ From the Brascamp-Lieb inequality \cite[Theorem 5.1]{brascampleb}, 
\begin{align}\label{add:eq-9}
	\Bigl\la \Bigl|\frac{w\cdot\sigma^1}{N}-\frac{w\cdot \la \sigma^1\ra_t}{N}\Bigr|^4\Bigr\ra_t&\leq \frac{K\|w\|^4}{N^4},
\end{align}
where $K$ is a constant independent of $w,N,t.$ Therefore, letting $w=(1,\ldots, 1)$ implies
\begin{align}\label{add:eq-12}
	\bigl\la \bigl|M_1-\la M_1\ra_t\bigr|^4\bigr\ra_t&\leq \frac{K\|w\|^4}{N^4}=\frac{K}{N^2}.
\end{align}
In addition, writing
\begin{align*}
	R_{1,2}-\la R_{1,2}\ra_t&=\frac{\sigma^1\cdot\sigma^2}{N}-\frac{\la \sigma^1\ra_t\cdot\sigma^2}{N}+\frac{\la \sigma^1\ra_t\cdot\sigma^2}{N}-\frac{\la \sigma^1\ra_t\cdot \la \sigma^2\ra_t}{N},
\end{align*}
we can apply \eqref{add:eq-9} with the choices $w=\sigma^2$ and $w=\la \sigma^1\ra_t$ to get
\begin{align*}
	\Bigl\la \Bigl|\frac{\sigma^1\cdot\sigma^2}{N}-\frac{\la \sigma^1\ra_t\cdot \sigma^2}{N}\Bigr|^4\Bigr\ra_t&\leq \frac{K\la \|\sigma^2\|^4\ra_t}{N^4}= \frac{K\la \|\sigma^1\|^4\ra_t}{N^4}
\end{align*}
and
\begin{align*}
		\Bigl\la \Bigl|\frac{\la \sigma^1\ra_t\cdot\sigma^2}{N}-\frac{\la \sigma^1\ra_t\cdot \la \sigma^2\ra_t}{N}\Bigr|^4\Bigr\ra_t&\leq \frac{K \|\la\sigma^1\ra_t\|^4}{N^4}\leq \frac{K \la\|\sigma^1\|^4\ra_t}{N^4}.
\end{align*}
Putting these together and applying \eqref{add:eq-10} with $k=2$, we obtain that for all $0\leq t\leq 1$ and $N\geq 1,$
\begin{align}\label{add:eq-11}
	\e\bigl\la \bigl|R_{1,2}-\la R_{1,2}\ra_t\bigr|^4\bigr\ra_t&\leq \frac{32K\e\la\|\sigma^1\|^4\ra_t}{N^4}\leq \frac{K'}{N^2}
\end{align}
for some constant $K'$ independent of $t,N.$
Plugging \eqref{add:eq-12} and \eqref{add:eq-11} into \eqref{add:eq-8} validates \eqref{add:eq-7}.
\end{proof}

Now we proceed to establish \eqref{add:thm1:eq1} and \eqref{add:thm1:eq2}. From Theorem \ref{thm:conv_to_solution} and \eqref{add:eq-3}, 
\begin{align*}
	\lim_{N\to\infty}\e \la R_{1,2}\ra_h&=q
\end{align*}
and
\begin{align*}
\frac{1}{N}\e\|\la x\ra^\sim-x^*\|^2&=\frac{1}{N}\e \|\la \sigma\ra_{h_N}\|^2=\e\la R_{1,2}\ra_{h_N}.
\end{align*}
These and \eqref{add:eq-7} together yields \eqref{add:thm1:eq2}. 

As for \eqref{add:thm1:eq1}, from \eqref{add:eq-5.1}, we can write
\begin{align*}
	\tilde F_N&=-\kappa \gamma_N+F_N(x^*),
\end{align*}
where 
\begin{align*}
F_N(x^*):=\frac{1}{N}\ln\int_{\mathbb{R}^N}\exp \Big(\sum_{k\leq M}u\Big((\bar G O^\intercal\sigma)_k-z_k\Big)-h_N\sum_{i\le N}\sigma_i-\kappa \|\sigma\|^2\Big)\,d\sigma
\end{align*}
and $O$ is an orthonormal matrix satisfying $Ox^*=N^{-1/2}{\|x^*\|}\boldsymbol{1}_N.$ Write
\begin{align}
	\begin{split}\label{add:eq-13}
\tilde F_N-\e \tilde F_N&=-\kappa(\gamma_N-\e\gamma_N)\\
&\qquad+F_N(x^*)-\e [F_N(x^*)\mid x^*]\\
&\qquad+\e [F_N(x^*)\mid x^*]-\e  F_N(x^*).
\end{split}
\end{align}
Here the first term on the right-hand side vanishes in $L^2$ norm by the given assumption. For the second term, 
observe that conditionally on $x^*$, $F_N(x^*)$ is the same as $F_N(h_N)$ in distribution. 
%In addition, since $h_N$ is independent of $(G,z)$ and the term $h_N\sum_{i\le N}\sigma_i$ in the Hamiltonian associated with $F_N(h_N)$ is linear, it can be checked that the proof of Lemma \ref{lem:conc_fe} is independent of whether $h$ is deterministic or randomized so that we can still apply the conclusion of 
Lemma \ref{lem:conc_fe} (with $n=1$ and $\lambda=\eta=0$) allows us to write that with probability one,
\begin{align*}
	\mbox{Var}(F_N(x^*)\mid x^*)=\mbox{Var}(F_N(h_N)\mid h_N)&\leq \frac{K(1+h_N^{4})}{N},
\end{align*}
where $K$ is independent of $N.$ Hence, we can bound the second term by
\begin{align*}
	\e\bigl|F_N(x^*)-\e [F_N(x^*)\mid x^*]\bigr|&\le \big(\e\bigl|F_N(x^*)-\e [F_N(x^*)\mid x^*]\bigr|^2\big)^{1/2}\\
	&=\big(\e\,	\mbox{Var}(F_N(x^*)\mid x^*)\big)^{1/2}\leq \sqrt\frac{K(1+\e h_N^{4})}{ N}.
\end{align*}
Lastly, write
\begin{align*}
	\e[F_N(x^*)\mid x^*]-\e F_N(x^*)&=\e[F_N(h_N)\mid h_N]-\e F_N(h_N)\\
	&=\e[F_N(h_N)-F_N(h)\mid h_N]+\bigl(\e F_N(h)-\e F_N(h_N) \bigr).
\end{align*}
We bound the third term in \eqref{add:eq-13} by using the Jensen inequality,
\begin{align*}
	\e\bigl|\e[F_N(x^*)\mid x^*]-\e F_N(x^*) \bigr|&\leq \e\bigl|\e[F_N(h_N)-F_N(h)\mid h_N]\bigr|+\bigl|\e F_N(h)-\e F_N(h_N)\bigr|\\
	&\leq \e\bigl|F_N(h_N)-F_N(h)\bigr|\to 0,
\end{align*}
where the last limit used \eqref{add:eq-6}. Combined with Theorem \ref{ThmMain} and the fact that $\e h_N^{4}$ is bounded by assumption \eqref{convNorm}, this completes our proof for \eqref{add:thm1:eq1}.

%%%%%%%%%%%%%%%%%%%%%%%%%%%%%%%%%%%%%%%%%%%%%%
%% Funding information, if any,             %%
%% should be provided in the                %%
%% funding section.                         %%
%%%%%%%%%%%%%%%%%%%%%%%%%%%%%%%%%%%%%%%%%%%%%%
\begin{funding}
W.-K.C. was partly supported by NSF DMS-17-52184. D.P. was partially supported by NSERC and Simons Fellowship. 
\end{funding}

%%%%%%%%%%%%%%%%%%%%%%%%%%%%%%%%%%%%%%%%%%%%%%%%%%%%%%%%%%%%%
%%                  The Bibliography                       %%
%%                                                         %%
%%  imsart-???.bst  will be used to                        %%
%%  create a .BBL file for submission.                     %%
%%                                                         %%
%%  Note that the displayed Bibliography will not          %%
%%  necessarily be rendered by Latex exactly as specified  %%
%%  in the online Instructions for Authors.                %%
%%                                                         %%
%%  MR numbers will be added by VTeX.                      %%
%%                                                         %%
%%  Use \cite{...} to cite references in text.             %%
%%                                                         %%
%%%%%%%%%%%%%%%%%%%%%%%%%%%%%%%%%%%%%%%%%%%%%%%%%%%%%%%%%%%%%

%% if your bibliography is in bibtex format, uncomment commands:
\bibliographystyle{imsart-number} % Style BST file (imsart-number.bst or imsart-nameyear.bst)
\bibliography{refs}       % Bibliography file (usually '*.bib')

\begin{thebibliography}{54}
% BibTex style file: imsart-number.bst, 2017-11-03
% Default style options (sort=1,type=number).
% Used options (sort=1,type=number).

\bibitem{abramovich2010map}
\begin{barticle}[author]
\bauthor{\bsnm{Abramovich},~\bfnm{Felix}\binits{F.}} \AND
  \bauthor{\bsnm{Grinshtein},~\bfnm{Vadim}\binits{V.}}
(\byear{2010}).
\btitle{MAP model selection in Gaussian regression}.
\bjournal{Electronic Journal of Statistics}
\bvolume{4}
\bpages{932--949}.
\end{barticle}
\endbibitem

\bibitem{advani2016statistical}
\begin{barticle}[author]
\bauthor{\bsnm{Advani},~\bfnm{M.}\binits{M.}} \AND
  \bauthor{\bsnm{Ganguli},~\bfnm{S.}\binits{S.}}
(\byear{2016}).
\btitle{Statistical mechanics of optimal convex inference in high dimensions}.
\bjournal{Physical Review X}
\bvolume{6}
\bpages{031034}.
\end{barticle}
\endbibitem

\bibitem{advani2016equivalence}
\begin{binproceedings}[author]
\bauthor{\bsnm{Advani},~\bfnm{M.}\binits{M.}} \AND
  \bauthor{\bsnm{Ganguli},~\bfnm{S.}\binits{S.}}
(\byear{2016}).
\btitle{An equivalence between high dimensional Bayes optimal inference and
  M-estimation}.
In \bbooktitle{Advances in Neural Information Processing Systems}
\bpages{3378--3386}.
\end{binproceedings}
\endbibitem

\bibitem{arias2014estimation}
\begin{barticle}[author]
\bauthor{\bsnm{Arias-Castro},~\bfnm{Ery}\binits{E.}} \AND
  \bauthor{\bsnm{Lounici},~\bfnm{Karim}\binits{K.}}
(\byear{2014}).
\btitle{Estimation and variable selection with exponential weights}.
\bjournal{Electronic Journal of Statistics}
\bvolume{8}
\bpages{328--354}.
\end{barticle}
\endbibitem

\bibitem{aubin2020generalization}
\begin{binproceedings}[author]
\bauthor{\bsnm{Aubin},~\bfnm{B.}\binits{B.}},
  \bauthor{\bsnm{Krzakala},~\bfnm{F.}\binits{F.}},
  \bauthor{\bsnm{Lu},~\bfnm{Y.}\binits{Y.}} \AND
  \bauthor{\bsnm{Zdeborov\'{a}},~\bfnm{L.}\binits{L.}}
(\byear{2020}).
\btitle{Generalization error in high-dimensional perceptrons: Approaching Bayes
  error with convex optimization}.
In \bbooktitle{Advances in Neural Information Processing Systems}
\bvolume{33}.
\end{binproceedings}
\endbibitem

\bibitem{banerjee2021bayesian}
\begin{barticle}[author]
\bauthor{\bsnm{Banerjee},~\bfnm{Sayantan}\binits{S.}},
  \bauthor{\bsnm{Castillo},~\bfnm{Isma{\"e}l}\binits{I.}} \AND
  \bauthor{\bsnm{Ghosal},~\bfnm{Subhashis}\binits{S.}}
(\byear{2021}).
\btitle{Bayesian inference in high-dimensional models}.
\bjournal{arXiv preprint arXiv:2101.04491}.
\end{barticle}
\endbibitem

\bibitem{barbier2016mutual}
\begin{binproceedings}[author]
\bauthor{\bsnm{Barbier},~\bfnm{J.}\binits{J.}},
  \bauthor{\bsnm{Dia},~\bfnm{M.}\binits{M.}},
  \bauthor{\bsnm{Macris},~\bfnm{N.}\binits{N.}} \AND
  \bauthor{\bsnm{Krzakala},~\bfnm{F.}\binits{F.}}
(\byear{2016}).
\btitle{The mutual information in random linear estimation}.
In \bbooktitle{2016 54th Annual Allerton Conference on Communication, Control,
  and Computing (Allerton)}
\bpages{625--632}.
\bpublisher{IEEE}.
\end{binproceedings}
\endbibitem

\bibitem{barbier2019optimal}
\begin{barticle}[author]
\bauthor{\bsnm{Barbier},~\bfnm{J.}\binits{J.}},
  \bauthor{\bsnm{Krzakala},~\bfnm{F.}\binits{F.}},
  \bauthor{\bsnm{Macris},~\bfnm{N.}\binits{N.}},
  \bauthor{\bsnm{Miolane},~\bfnm{L.}\binits{L.}} \AND
  \bauthor{\bsnm{Zdeborov{\'a}},~\bfnm{L.}\binits{L.}}
(\byear{2019}).
\btitle{Optimal errors and phase transitions in high-dimensional generalized
  linear models}.
\bjournal{Proceedings of the National Academy of Sciences}
\bvolume{116}
\bpages{5451--5460}.
\end{barticle}
\endbibitem

\bibitem{barbier2020mutual}
\begin{barticle}[author]
\bauthor{\bsnm{Barbier},~\bfnm{J.}\binits{J.}},
  \bauthor{\bsnm{Macris},~\bfnm{N.}\binits{N.}},
  \bauthor{\bsnm{Dia},~\bfnm{M.}\binits{M.}} \AND
  \bauthor{\bsnm{Krzakala},~\bfnm{F.}\binits{F.}}
(\byear{2020}).
\btitle{Mutual information and optimality of approximate message-passing in
  random linear estimation}.
\bjournal{IEEE Transactions on Information Theory}
\bvolume{66}
\bpages{4270--4303}.
\end{barticle}
\endbibitem

\bibitem{DBLP:conf/nips/BarbierMR20}
\begin{binproceedings}[author]
\bauthor{\bsnm{Barbier},~\bfnm{J.}\binits{J.}},
  \bauthor{\bsnm{Macris},~\bfnm{N.}\binits{N.}} \AND
  \bauthor{\bsnm{Rush},~\bfnm{C.}\binits{C.}}
(\byear{2020}).
\btitle{All-or-nothing statistical and computational phase transitions in
  sparse spiked matrix estimation}.
In \bbooktitle{Advances in Neural Information Processing Systems 33: Annual
  Conference on Neural Information Processing Systems 2020, NeurIPS 2020,
  December 6-12, 2020, virtual}.
\end{binproceedings}
\endbibitem

\bibitem{barbier2020strong}
\begin{barticle}[author]
\bauthor{\bsnm{Barbier},~\bfnm{Jean}\binits{J.}} \AND
  \bauthor{\bsnm{Panchenko},~\bfnm{Dmitry}\binits{D.}}
(\byear{2020}).
\btitle{Strong replica symmetry in high-dimensional optimal Bayesian
  inference}.
\bjournal{arXiv preprint arXiv:2005.03115}.
\end{barticle}
\endbibitem

\bibitem{bean2013optimal}
\begin{barticle}[author]
\bauthor{\bsnm{Bean},~\bfnm{D.}\binits{D.}},
  \bauthor{\bsnm{Bickel},~\bfnm{P.}\binits{P.}},
  \bauthor{\bsnm{El~Karoui},~\bfnm{N.}\binits{N.}} \AND
  \bauthor{\bsnm{Yu},~\bfnm{B.}\binits{B.}}
(\byear{2013}).
\btitle{Optimal M-estimation in high-dimensional regression}.
\bjournal{Proceedings of the National Academy of Sciences}
\bvolume{110}
\bpages{14563--14568}.
\end{barticle}
\endbibitem

\bibitem{bottolo2010evolutionary}
\begin{barticle}[author]
\bauthor{\bsnm{Bottolo},~\bfnm{Leonard}\binits{L.}} \AND
  \bauthor{\bsnm{Richardson},~\bfnm{Sylvia}\binits{S.}}
(\byear{2010}).
\btitle{Evolutionary stochastic search for Bayesian model exploration}.
\bjournal{Bayesian Analysis}
\bvolume{5}
\bpages{583--618}.
\end{barticle}
\endbibitem

\bibitem{bradic2016robustness}
\begin{barticle}[author]
\bauthor{\bsnm{Bradic},~\bfnm{J.}\binits{J.}}
(\byear{2016}).
\btitle{Robustness in sparse high-dimensional linear models: relative
  efficiency and robust approximate message passing}.
\bjournal{Electronic Journal of Statistics}
\bvolume{10}
\bpages{3894--3944}.
\end{barticle}
\endbibitem

\bibitem{brascampleb}
\begin{barticle}[author]
\bauthor{\bsnm{Brascamp},~\bfnm{H.~J.}\binits{H.~J.}} \AND
  \bauthor{\bsnm{Lieb},~\bfnm{E.~H.}\binits{E.~H.}}
(\byear{1976}).
\btitle{On extensions of the {B}runn-{M}inkowski and {P}r\'ekopa-Leindler
  theorems, including inequalities for log concave functions, and with an
  application to the diffusion equation}.
\bjournal{Journal of functional Analysis}
\bvolume{22}
\bpages{366--389}.
\end{barticle}
\endbibitem

\bibitem{castillo2015bayesian}
\begin{barticle}[author]
\bauthor{\bsnm{Castillo},~\bfnm{Isma{\"e}l}\binits{I.}},
  \bauthor{\bsnm{Schmidt-Hieber},~\bfnm{Johannes}\binits{J.}} \AND
  \bauthor{\bparticle{Van~der} \bsnm{Vaart},~\bfnm{Aad}\binits{A.}}
(\byear{2015}).
\btitle{Bayesian linear regression with sparse priors}.
\bjournal{The Annals of Statistics}
\bvolume{43}
\bpages{1986--2018}.
\end{barticle}
\endbibitem

\bibitem{celentano2019fundamental}
\begin{barticle}[author]
\bauthor{\bsnm{Celentano},~\bfnm{M.}\binits{M.}} \AND
  \bauthor{\bsnm{Montanari},~\bfnm{A.}\binits{A.}}
(\byear{2019}).
\btitle{Fundamental barriers to high-dimensional regression with convex
  penalties}.
\bjournal{arXiv preprint arXiv:1903.10603}.
\end{barticle}
\endbibitem

\bibitem{contucci2009spin}
\begin{bincollection}[author]
\bauthor{\bsnm{Contucci},~\bfnm{P.}\binits{P.}},
  \bauthor{\bsnm{Giardina},~\bfnm{C.}\binits{C.}} \AND
  \bauthor{\bsnm{Nishimori},~\bfnm{H.}\binits{H.}}
(\byear{2009}).
\btitle{Spin glass identities and the {N}ishimori line}.
In \bbooktitle{Spin Glasses: Statics and Dynamics}
\bpages{103--121}.
\bpublisher{Springer}.
\end{bincollection}
\endbibitem

\bibitem{Dembo1998LargeDT}
\begin{bbook}[author]
\bauthor{\bsnm{Dembo},~\bfnm{Amir}\binits{A.}} \AND
  \bauthor{\bsnm{Zeitouni},~\bfnm{Ofer}\binits{O.}}
(\byear{1998}).
\btitle{Large Deviations Techniques and Applications, 2nd edition}.
\bpublisher{Springer}.
\end{bbook}
\endbibitem

\bibitem{donoho2016high}
\begin{barticle}[author]
\bauthor{\bsnm{Donoho},~\bfnm{D.}\binits{D.}} \AND
  \bauthor{\bsnm{Montanari},~\bfnm{A.}\binits{A.}}
(\byear{2016}).
\btitle{High dimensional robust m-estimation: asymptotic variance via
  approximate message passing}.
\bjournal{Probability Theory and Related Fields}
\bvolume{166}
\bpages{935--969}.
\end{barticle}
\endbibitem

\bibitem{el2018impact}
\begin{barticle}[author]
\bauthor{\bsnm{El~Karoui},~\bfnm{N.}\binits{N.}}
(\byear{2018}).
\btitle{On the impact of predictor geometry on the performance on
  high-dimensional ridge-regularized generalized robust regression estimators}.
\bjournal{Probability Theory and Related Fields}
\bvolume{170}
\bpages{95--175}.
\end{barticle}
\endbibitem

\bibitem{el2013robust}
\begin{barticle}[author]
\bauthor{\bsnm{El~Karoui},~\bfnm{N.}\binits{N.}},
  \bauthor{\bsnm{Bean},~\bfnm{D.}\binits{D.}},
  \bauthor{\bsnm{Bickel},~\bfnm{P.~J.}\binits{P.~J.}},
  \bauthor{\bsnm{Lim},~\bfnm{C.}\binits{C.}} \AND
  \bauthor{\bsnm{Yu},~\bfnm{B.}\binits{B.}}
(\byear{2013}).
\btitle{On robust regression with high-dimensional predictors}.
\bjournal{Proceedings of the National Academy of Sciences}
\bvolume{110}
\bpages{14557--14562}.
\end{barticle}
\endbibitem

\bibitem{david2017high}
\begin{binproceedings}[author]
\bauthor{\bsnm{Gamarnik},~\bfnm{D.}\binits{D.}} \AND
  \bauthor{\bsnm{Zadik},~\bfnm{I.}\binits{I.}}
(\byear{2017}).
\btitle{High dimensional regression with binary coefficients. estimating
  squared error and a phase transtition}.
In \bbooktitle{Conference on Learning Theory}
\bpages{948--953}.
\end{binproceedings}
\endbibitem

\bibitem{george2000calibration}
\begin{barticle}[author]
\bauthor{\bsnm{George},~\bfnm{EdwardI}\binits{E.}} \AND
  \bauthor{\bsnm{Foster},~\bfnm{Dean~P}\binits{D.~P.}}
(\byear{2000}).
\btitle{Calibration and empirical Bayes variable selection}.
\bjournal{Biometrika}
\bvolume{87}
\bpages{731--747}.
\end{barticle}
\endbibitem

\bibitem{george2000variable}
\begin{barticle}[author]
\bauthor{\bsnm{George},~\bfnm{Edward~I}\binits{E.~I.}}
(\byear{2000}).
\btitle{The variable selection problem}.
\bjournal{Journal of the American Statistical Association}
\bvolume{95}
\bpages{1304--1308}.
\end{barticle}
\endbibitem

\bibitem{gerbelot2020asymptotic2}
\begin{binproceedings}[author]
\bauthor{\bsnm{Gerbelot},~\bfnm{C.}\binits{C.}},
  \bauthor{\bsnm{Abbara},~\bfnm{A.}\binits{A.}} \AND
  \bauthor{\bsnm{Krzakala},~\bfnm{F.}\binits{F.}}
(\byear{2020}).
\btitle{Asymptotic errors for convex penalized linear regression beyond
  gaussian matrices}.
In \bbooktitle{Proceedings of Machine Learning Research: Conference on Learning
  Theory}
\bvolume{, 125}
\bpages{1682-1713}.
\end{binproceedings}
\endbibitem

\bibitem{gerbelot2020asymptotic}
\begin{barticle}[author]
\bauthor{\bsnm{Gerbelot},~\bfnm{C.}\binits{C.}},
  \bauthor{\bsnm{Abbara},~\bfnm{A.}\binits{A.}} \AND
  \bauthor{\bsnm{Krzakala},~\bfnm{F.}\binits{F.}}
(\byear{2020}).
\btitle{Asymptotic errors for teacher-Student Convex Generalized Linear Models
  (or: how to Prove Kabashima's Replica Formula)}.
\bjournal{arXiv preprint arXiv:2006.06581}.
\end{barticle}
\endbibitem

\bibitem{ishwaran2005spike}
\begin{barticle}[author]
\bauthor{\bsnm{Ishwaran},~\bfnm{Hemant}\binits{H.}} \AND
  \bauthor{\bsnm{Rao},~\bfnm{J~Sunil}\binits{J.~S.}}
(\byear{2005}).
\btitle{Spike and slab variable selection: frequentist and Bayesian
  strategies}.
\bjournal{The Annals of Statistics}
\bvolume{33}
\bpages{730--773}.
\end{barticle}
\endbibitem

\bibitem{DBLP:conf/nips/LuneauBM20}
\begin{binproceedings}[author]
\bauthor{\bsnm{Luneau},~\bfnm{C.}\binits{C.}},
  \bauthor{\bsnm{Barbier},~\bfnm{J.}\binits{J.}} \AND
  \bauthor{\bsnm{Macris},~\bfnm{N.}\binits{N.}}
(\byear{2020}).
\btitle{Information theoretic limits of learning a sparse rule}.
In \bbooktitle{Advances in Neural Information Processing Systems 33: Annual
  Conference on Neural Information Processing Systems 2020, NeurIPS 2020,
  December 6-12, 2020, virtual}.
\end{binproceedings}
\endbibitem

\bibitem{martin2017empirical}
\begin{barticle}[author]
\bauthor{\bsnm{Martin},~\bfnm{Ryan}\binits{R.}},
  \bauthor{\bsnm{Mess},~\bfnm{Raymond}\binits{R.}} \AND
  \bauthor{\bsnm{Walker},~\bfnm{Stephen~G}\binits{S.~G.}}
(\byear{2017}).
\btitle{Empirical Bayes posterior concentration in sparse high-dimensional
  linear models}.
\bjournal{Bernoulli}
\bvolume{23}
\bpages{1822--1847}.
\end{barticle}
\endbibitem

\bibitem{MezardMontanari09}
\begin{bbook}[author]
\bauthor{\bsnm{M\'ezard},~\bfnm{M.}\binits{M.}} \AND
  \bauthor{\bsnm{Montanari},~\bfnm{A.}\binits{A.}}
(\byear{2009}).
\btitle{Information, Physics and Computation}.
\bpublisher{Oxford University Press}.
\end{bbook}
\endbibitem

\bibitem{MezardParisi87b}
\begin{bbook}[author]
\bauthor{\bsnm{M{\'e}zard},~\bfnm{M.}\binits{M.}},
  \bauthor{\bsnm{Parisi},~\bfnm{G.}\binits{G.}} \AND
  \bauthor{\bsnm{Virasoro},~\bfnm{M.~A.}\binits{M.~A.}}
(\byear{1987}).
\btitle{Spin-Glass Theory and Beyond}
\bvolume{9}.
\bpublisher{World Scientific}, \baddress{Singapore}.
\end{bbook}
\endbibitem

\bibitem{mitchell1988bayesian}
\begin{barticle}[author]
\bauthor{\bsnm{Mitchell},~\bfnm{Toby~J}\binits{T.~J.}} \AND
  \bauthor{\bsnm{Beauchamp},~\bfnm{John~J}\binits{J.~J.}}
(\byear{1988}).
\btitle{Bayesian variable selection in linear regression}.
\bjournal{Journal of the american statistical association}
\bvolume{83}
\bpages{1023--1032}.
\end{barticle}
\endbibitem

\bibitem{mukherjee2021variational}
\begin{barticle}[author]
\bauthor{\bsnm{Mukherjee},~\bfnm{S.}\binits{S.}} \AND
  \bauthor{\bsnm{Sen},~\bfnm{S.}\binits{S.}}
(\byear{2021}).
\btitle{Variational Inference in high-dimensional linear regression}.
\bjournal{arXiv preprint arXiv:2104.12232}.
\end{barticle}
\endbibitem

\bibitem{DBLP:conf/nips/Niles-WeedZ20}
\begin{binproceedings}[author]
\bauthor{\bsnm{Niles-Weed},~\bfnm{J.}\binits{J.}} \AND
  \bauthor{\bsnm{Zadik},~\bfnm{I.}\binits{I.}}
(\byear{2020}).
\btitle{The All-or-Nothing Phenomenon in Sparse Tensor {PCA}}.
In \bbooktitle{Advances in Neural Information Processing Systems 33: Annual
  Conference on Neural Information Processing Systems 2020, NeurIPS 2020,
  December 6-12, 2020, virtual}.
\end{binproceedings}
\endbibitem

\bibitem{niles2021all}
\begin{barticle}[author]
\bauthor{\bsnm{Niles-Weed},~\bfnm{J.}\binits{J.}} \AND
  \bauthor{\bsnm{Zadik},~\bfnm{I.}\binits{I.}}
(\byear{2021}).
\btitle{It was "all" for "nothing": sharp phase transitions for noiseless
  discrete channels}.
\bjournal{arXiv preprint arXiv:2102.12422}.
\end{barticle}
\endbibitem

\bibitem{NishimoriBook01}
\begin{bbook}[author]
\bauthor{\bsnm{Nishimori},~\bfnm{H.}\binits{H.}}
(\byear{2001}).
\btitle{Statistical Physics of Spin Glasses and Information Processing: An
  Introduction}.
\bpublisher{Oxford University Press}.
\end{bbook}
\endbibitem

\bibitem{panchenko2013sherrington}
\begin{bbook}[author]
\bauthor{\bsnm{Panchenko},~\bfnm{D.}\binits{D.}}
(\byear{2013}).
\btitle{The Sherrington-Kirkpatrick model}.
\bpublisher{Springer Science \& Business Media}.
\end{bbook}
\endbibitem

\bibitem{ray2021variational}
\begin{barticle}[author]
\bauthor{\bsnm{Ray},~\bfnm{Kolyan}\binits{K.}} \AND
  \bauthor{\bsnm{Szab{\'o}},~\bfnm{Botond}\binits{B.}}
(\byear{2021}).
\btitle{Variational Bayes for high-dimensional linear regression with sparse
  priors}.
\bjournal{Journal of the American Statistical Association}
\bpages{1--12}.
\end{barticle}
\endbibitem

\bibitem{reeves2016replica}
\begin{binproceedings}[author]
\bauthor{\bsnm{Reeves},~\bfnm{G.}\binits{G.}} \AND
  \bauthor{\bsnm{Pfister},~\bfnm{H.~D.}\binits{H.~D.}}
(\byear{2016}).
\btitle{The replica-symmetric prediction for compressed sensing with Gaussian
  matrices is exact}.
In \bbooktitle{2016 IEEE International Symposium on Information Theory (ISIT)}
\bpages{665--669}.
\bpublisher{IEEE}.
\end{binproceedings}
\endbibitem

\bibitem{reeves2019all}
\begin{binproceedings}[author]
\bauthor{\bsnm{Reeves},~\bfnm{G.}\binits{G.}},
  \bauthor{\bsnm{Xu},~\bfnm{J.}\binits{J.}} \AND
  \bauthor{\bsnm{Zadik},~\bfnm{I.}\binits{I.}}
(\byear{2019}).
\btitle{The all-or-nothing phenomenon in sparse linear regression}.
In \bbooktitle{Conference on Learning Theory}
\bpages{2652--2663}.
\bpublisher{PMLR}.
\end{binproceedings}
\endbibitem

\bibitem{reeves2019all_2}
\begin{binproceedings}[author]
\bauthor{\bsnm{Reeves},~\bfnm{G.}\binits{G.}},
  \bauthor{\bsnm{Xu},~\bfnm{J.}\binits{J.}} \AND
  \bauthor{\bsnm{Zadik},~\bfnm{I.}\binits{I.}}
(\byear{2019}).
\btitle{All-or-nothing phenomena: From single-letter to high dimensions}.
In \bbooktitle{2019 IEEE 8th International Workshop on Computational Advances
  in Multi-Sensor Adaptive Processing (CAMSAP)}
\bpages{654--658}.
\bpublisher{IEEE}.
\end{binproceedings}
\endbibitem

\bibitem{scott2010bayes}
\begin{barticle}[author]
\bauthor{\bsnm{Scott},~\bfnm{James~G}\binits{J.~G.}} \AND
  \bauthor{\bsnm{Berger},~\bfnm{James~O}\binits{J.~O.}}
(\byear{2010}).
\btitle{Bayes and empirical-Bayes multiplicity adjustment in the
  variable-selection problem}.
\bjournal{The Annals of Statistics}
\bpages{2587--2619}.
\end{barticle}
\endbibitem

\bibitem{shcherbina2003rigorous}
\begin{barticle}[author]
\bauthor{\bsnm{Shcherbina},~\bfnm{M.}\binits{M.}} \AND
  \bauthor{\bsnm{Tirozzi},~\bfnm{B.}\binits{B.}}
(\byear{2003}).
\btitle{Rigorous solution of the Gardner problem}.
\bjournal{Communications in mathematical physics}
\bvolume{234}
\bpages{383--422}.
\end{barticle}
\endbibitem

\bibitem{sur2019modern}
\begin{barticle}[author]
\bauthor{\bsnm{Sur},~\bfnm{P.}\binits{P.}} \AND
  \bauthor{\bsnm{Cand{\`e}s},~\bfnm{E.~J.}\binits{E.~J.}}
(\byear{2019}).
\btitle{A modern maximum-likelihood theory for high-dimensional logistic
  regression}.
\bjournal{Proceedings of the National Academy of Sciences}
\bvolume{116}
\bpages{14516--14525}.
\end{barticle}
\endbibitem

\bibitem{taheri2020fundamental}
\begin{barticle}[author]
\bauthor{\bsnm{Taheri},~\bfnm{H.}\binits{H.}},
  \bauthor{\bsnm{Pedarsani},~\bfnm{R.}\binits{R.}} \AND
  \bauthor{\bsnm{Thrampoulidis},~\bfnm{C.}\binits{C.}}
(\byear{2020}).
\btitle{Fundamental Limits of Ridge-Regularized Empirical Risk Minimization in
  High Dimensions}.
\bjournal{arXiv preprint arXiv:2006.08917}.
\end{barticle}
\endbibitem

\bibitem{9173999}
\begin{binproceedings}[author]
\bauthor{\bsnm{Takahashi},~\bfnm{T.}\binits{T.}} \AND
  \bauthor{\bsnm{Kabashima},~\bfnm{Y.}\binits{Y.}}
(\byear{2020}).
\btitle{Macroscopic Analysis of Vector Approximate Message Passing in a Model
  Mismatch Setting}.
In \bbooktitle{2020 IEEE International Symposium on Information Theory (ISIT)}
\bpages{1403-1408}.
\bdoi{10.1109/ISIT44484.2020.9173999}
\end{binproceedings}
\endbibitem

\bibitem{Talagrand2011spina}
\begin{bbook}[author]
\bauthor{\bsnm{Talagrand},~\bfnm{M.}\binits{M.}}
(\byear{2011}).
\btitle{Mean Field Models for Spin Glasses. Volume I: Basic Examples}.
\bpublisher{Springer Verlag}.
\end{bbook}
\endbibitem

\bibitem{thrampoulidis2018precise}
\begin{barticle}[author]
\bauthor{\bsnm{Thrampoulidis},~\bfnm{C.}\binits{C.}},
  \bauthor{\bsnm{Abbasi},~\bfnm{E.}\binits{E.}} \AND
  \bauthor{\bsnm{Hassibi},~\bfnm{B.}\binits{B.}}
(\byear{2018}).
\btitle{Precise Error Analysis of Regularized $ M $-Estimators in High
  Dimensions}.
\bjournal{IEEE Transactions on Information Theory}
\bvolume{64}
\bpages{5592--5628}.
\end{barticle}
\endbibitem

\bibitem{pmlr-v40-Thrampoulidis15}
\begin{binproceedings}[author]
\bauthor{\bsnm{Thrampoulidis},~\bfnm{C.}\binits{C.}},
  \bauthor{\bsnm{Oymak},~\bfnm{S.}\binits{S.}} \AND
  \bauthor{\bsnm{Hassibi},~\bfnm{B.}\binits{B.}}
(\byear{2015}).
\btitle{Regularized Linear Regression: A Precise Analysis of the Estimation
  Error}.
In \bbooktitle{Proceedings of The 28th Conference on Learning Theory}.
\bseries{Proceedings of Machine Learning Research}
\bvolume{40}
\bpages{1683--1709}.
\end{binproceedings}
\endbibitem

\bibitem{touchette2009large}
\begin{barticle}[author]
\bauthor{\bsnm{Touchette},~\bfnm{Hugo}\binits{H.}}
(\byear{2009}).
\btitle{The large deviation approach to statistical mechanics}.
\bjournal{Physics Reports}
\bvolume{478}
\bpages{1--69}.
\end{barticle}
\endbibitem

\bibitem{vershynin2010introduction}
\begin{barticle}[author]
\bauthor{\bsnm{Vershynin},~\bfnm{R.}\binits{R.}}
(\byear{2010}).
\btitle{Introduction to the non-asymptotic analysis of random matrices}.
\bjournal{arXiv preprint arXiv:1011.3027}.
\end{barticle}
\endbibitem

\bibitem{wasserman2004all}
\begin{bbook}[author]
\bauthor{\bsnm{Wasserman},~\bfnm{Larry}\binits{L.}}
(\byear{2004}).
\btitle{All of statistics: a concise course in statistical inference}
\bvolume{26}.
\bpublisher{Springer}.
\end{bbook}
\endbibitem

\bibitem{yuan2005efficient}
\begin{barticle}[author]
\bauthor{\bsnm{Yuan},~\bfnm{Ming}\binits{M.}} \AND
  \bauthor{\bsnm{Lin},~\bfnm{Yi}\binits{Y.}}
(\byear{2005}).
\btitle{Efficient empirical Bayes variable selection and estimation in linear
  models}.
\bjournal{Journal of the American Statistical Association}
\bvolume{100}
\bpages{1215--1225}.
\end{barticle}
\endbibitem

\end{thebibliography}

%% or include bibliography directly:
% \begin{thebibliography}{4}
% \bibliography{refs} 
% \end{thebibliography}

\end{document}